\documentclass[oneside,12pt,a4paper]{amsart}

\usepackage{lmodern}
\usepackage[T1]{fontenc}
\usepackage[utf8]{inputenc}
\usepackage[english]{babel} %
\usepackage[doi=false,isbn=false,style=alphabetic,sorting=nyt,backend=biber, 
maxnames=5,maxalphanames=5,firstinits=true]{biblatex}
\bibliography{Bibliography}

\usepackage{etoolbox}
\patchcmd{\subsection}{-.5em}{.5em}{}{}

\usepackage{array}		  %
\usepackage{graphicx}   %
\usepackage{enumerate}  %
\usepackage[hidelinks]{hyperref}   %
\usepackage[capitalise,noabbrev]{cleveref}
\usepackage{ifthen}     %
\usepackage{enumitem}   %

\usepackage{amsmath}
\usepackage{amssymb}
\usepackage{mathrsfs}   %
\usepackage{mathtools}	%
\mathtoolsset{centercolon}	%
\usepackage{tikz-cd}    %
\usetikzlibrary{babel}  %
\usepackage{stmaryrd}   %

\usepackage{tikz}				%
\usepackage{tikz-3dplot}%
\usetikzlibrary{calc} %
\usetikzlibrary{arrows}	%

\newtheorem{thm}{Theorem}[section]
\newtheorem{prop}[thm]{Proposition}
\newtheorem{lemma}[thm]{Lemma}
\newtheorem{cor}[thm]{Corollary}
\newtheorem{conjecture}[thm]{Conjecture}
\newtheorem*{conjecture*}{Conjecture}

\newtheorem{question}[thm]{Question}

\newtheorem*{question*}{Question}
\newtheorem{definition}[thm]{Definition}

\theoremstyle{definition}
\newtheorem{remark}[thm]{Remark}

\newtheorem{ex}[thm]{Example}

\numberwithin{equation}{section}
\numberwithin{figure}{section}
\newcommand{\N}{\mathbb{N}}
\newcommand{\A}{\mathbb{A}}
\renewcommand{\P}{\mathbb{P}}
\newcommand{\R}{\mathbb{R}}
\newcommand{\C}{\mathbb{C}}

\newcommand{\Z}{\mathbb{Z}}

\renewcommand{\O}{\mathcal{O}}
\newcommand{\Ell}{\mathscr{L}}

\newcommand\restr[2]{{
    \left.\kern-\nulldelimiterspace 
    #1
    \right|_{#2} 
}}
\newcommand\supind[1]{{\smash[t]{(#1)}}}

\let\originalleft\left
\let\originalright\right
\renewcommand{\left}{\mathopen{}\mathclose\bgroup\originalleft}
\renewcommand{\right}{\aftergroup\egroup\originalright}

\hypersetup{bookmarksdepth=3}
\usepackage{fullpage}
\newenvironment{nouppercase}{%
  \renewcommand{\uppercasenonmath}[1]{}}{}

\newcommand{\defstyle}{\textbf}  %

\DeclareMathOperator{\Sym}{Sym}
\DeclareMathOperator{\Hom}{Hom}
\DeclareMathOperator{\conv}{conv}
\DeclareMathOperator{\Spec}{Spec}
\DeclareMathOperator{\Proj}{Proj}

\DeclareMathOperator{\GL}{GL}
\DeclareMathOperator{\id}{id}
\DeclareMathOperator{\ord}{ord}

\DeclareMathOperator{\diag}{diag}

\DeclareMathOperator{\Pic}{Pic}

\DeclareMathOperator{\im}{im}

\DeclareMathOperator{\codim}{codim}

\DeclareMathOperator{\Cox}{Cox}

\DeclareMathOperator{\lcm}{lcm}
\DeclareMathOperator{\Bs}{Bs}

\DeclareMathOperator{\Cl}{Cl}

\DeclareMathOperator{\Bl}{Bl}

\newcommand\xrightarrowinline[1]{\mathrel{\raisebox{-0.25pt}{\smash{$\xrightarrow{\smash[b]{\raisebox{-0.75pt}[0pt]{\(\scriptstyle #1\)}}}$}}}}
\newcommand\injTo{\xrightarrowinline{\rm{\,inj\,}}}
\newcommand\injDim[1]{\ensuremath{\gamma(#1)}}
\newcommand\nVec{\ensuremath{\mathbf{n}}}
\newcommand\dVec{\ensuremath{\mathbf{d}}}
\newcommand\gitFrac[2]{\ensuremath{{#1}/\!\!/{#2}}}

\title{\LARGE Injection dimensions of projective varieties}
\author{\large Paul Görlach}

\begin{document}
  
\begin{abstract}
  We explore injective morphisms from complex projective varieties $X$ to projective spaces $\P^s$ of small dimension.
  Based on connectedness theorems, we prove that the ambient dimension
  $s$ needs to be at least $2\dim X$ for all injections given by a linear subsystem of a strict power of a line bundle. 
  Using this, we give an example where the smallest ambient dimension cannot be attained from any embedding $X \hookrightarrow \P^n$ by linear projections.
  Our focus then lies on 
  $X = \P^{n_1} \times \ldots \times \P^{n_r}$, in which case there is a close connection to secant loci of Segre--Veronese varieties and the rank~$2$ geometry of partially symmetric tensors, as well as on $X = \P(q_0,\ldots,q_n)$, which is linked to separating invariants for
  representations of finite cyclic groups.
  We showcase three 
  techniques for constructing injections $X \to \P^{2\dim X}$ in specific cases.

\end{abstract}
  
\begin{nouppercase}
  \maketitle
\end{nouppercase}

\vspace{-0.7cm}

\setcounter{tocdepth}{2}
\begingroup \small
\tableofcontents
\endgroup

\vspace{-1cm}

\section{Introduction}
In Differential Geometry, the Whitney embedding theorem asserts that every $n$-dimen\-sional real smooth manifold admits a smooth embedding into $\R^{2n}$. The analogous question in Algebraic Geometry whether every $n$-dimensional complex projective variety admits a closed embedding into $2n$-dimensional complex projective space $\P^{2n}$ has a negative answer. However, when we relax the requirement on the morphism to $\P^{2n}$ from being a closed embedding to being injective (on the level of points), this is an open problem:

\begin{question} \label{qu:injectIntoTwiceDim}
  Does every $n$-dimensional connected 
  complex projective variety $X$ admit an injective morphism $X \injTo \P^{2n}$?
\end{question}

Among other results, we show in this article:
\begin{itemize}[label={$\scriptstyle \blacktriangleright$}]
  \item \cref{qu:injectIntoTwiceDim} can in general not be improved upon: For every $n \geq 3$, there exist connected $n$-dimensional projective varieties that cannot be injectively mapped to $\P^{2n-1}$. See \cref{ex:noSmallInjectionSingular} and \cref{ex:noSmallInjectionSmooth}.
  \item \cref{qu:injectIntoTwiceDim} has an affirmative answer for $X = \P^1 \times \P^1 \times \P^n$ and for weighted projective spaces of the form 
  $X = \P(1,q_1,\ldots,q_n)$ with $\lcm\{q_i,q_j\} = \lcm\{q_1,\ldots,q_n\}$ for all $i \neq j$. We provide injections into small ambient spaces in \cref{cor:P1P1Pm} and \cref{thm:weightedProjectiveConstr}.
\end{itemize}

In general, we define the \defstyle{injection dimension}
$\injDim{X} \in \N$ of a complex projective variety $X$ to be the smallest dimension of a projective space to which there exists an injective morphism from $X$. In a more refined setting, it is desirable to study the injection dimension $\injDim{X,\Ell}$ with respect to a fixed line bundle $\Ell$ on $X$, where we restrict to injective morphisms $X \injTo \P^s$ given by a linear subsystem of $|\Ell|$.

Already in the case of curves, the study of injection dimensions is interesting and largely open. Restricting to very ample line bundles, injections to $\P^2$ are given by cuspidal projections of embeddings of the curve. Very recent progress in this area has been made by \cite{BV18}, especially in the context of space curves lying on irreducible quadrics. More classical work dates back to \cite{Pie81}, where it was proved that general canonical curves of genus~$4$ do not admit cuspidal projections. In particular, this gives examples of curves $C$ with $\injDim{C,\omega_C} = 3 > 2 \dim C$, showing that refining \cref{qu:injectIntoTwiceDim} to injection dimensions with respect to all very ample line bundles cannot have an affirmative answer in general.

On the other hand, for all complex projective varieties $X$, a classical projection argument shows that $\injDim{X,\Ell} \leq 2 \dim X + 1$ for all line bundles $\Ell$ giving rise to injections. With techniques inspired by work on separating invariants \cite{DJ14,DJ16,Rei18}, we prove that there is very little room for improvement for line bundles admitting a root of some order.

\begin{thm}[\cref{thm:lowerBound}] \label{thm:lowerBoundIntro}
  Let $X$ be a complex projective variety and let $\Ell$ be a line bundle on $X$. Then $\injDim{X,\Ell^{\otimes k}} \geq 2\dim X$ for all $k \geq 2$, i.e., every injection $X \injTo \P^s$ given by a linear subsystem of $|\Ell^{\otimes k}|$ satisfies $s \geq 2\dim X$.
\end{thm}

This theorem vastly generalizes previous work in \cite[§5]{DJ16}, whose arguments amount to proving the above result for the special case $X = \P^{n_1} \times \ldots \times \P^{n_r}$ and $\Ell = \O(1,\ldots,1)$.\footnote{The result claimed in \cite[§5]{DJ16} is more general, but unfortunately incorrect. \cref{prop:tangentialInjections} gives a counterexample; see also the discussion at the beginning of \cref{sec:lowerBounds}.} We also refer to \cite[§6]{DJ16} for results on injective morphisms preserving toric structures from a viewpoint of separating invariants for representations of tori. 

In the setting of normal varieties with singularities, we provide an extension of \cref{thm:lowerBoundIntro} by giving a similar bound when the assumption of divisibility in the Picard group is replaced by divisibility in the class group, see \cref{thm:lowerBoundReflexive}. Applied to weighted projective spaces, this gives rise to the following result:

\begin{thm}[\cref{cor:weightedProjectiveBounds} and \cref{thm:weightedProjectiveConstr}] \label{thm:weightedProjectiveIntro}
  Consider a weighted projective space $X = \P(q_0,\ldots,q_n)$ with $\gcd(q_0,\ldots,\widehat{q_i},\ldots,q_n) = 1$ for all $i$, and let $\ell \geq 2$ be minimal such that $\lcm(q_{i_1},\ldots,q_{i_\ell}) = \lcm(q_0,\ldots,q_n)$ for all $i_1, \ldots, i_\ell$ distinct. Let $\Ell$ be the ample line bundle generating $\Pic(X)$. Then
    \[\injDim{X,\Ell^{\otimes k}} \geq
      \begin{cases} 
        n + \ell - 2 &\text{if } k = 1, \\ 2n &\text{if } k \geq 2. \end{cases}\]
  For $\ell \in \{2,3\}$ and $q_0 = 1$, equality holds.
\end{thm}

This generalizes the classification of injection dimensions for (non-weighted) projective spaces carried out in \cite{DJ14} and has analogues in the theory of separating invariants for actions of finite cyclic groups \cite{Duf08, DJ14}. From \cref{thm:weightedProjectiveIntro}, we deduce that for the weighted projective space $\P(1,6,10,15)$, the smallest injection dimension cannot be attained via linear projections starting from any embedding as a subvariety of projective space, see \cref{ex:minNotVeryAmple}.

We then focus on products of projective spaces: $X = \P^{n_1} \times \ldots \times \P^{n_r}$. 
In \cite{DJ16}, techniques from local cohomology were employed to bound their injection dimensions as follows:
  \[\injDim{\P^{n_1} \times \ldots \times \P^{n_r}} \geq 2({\textstyle \sum_{i=1}^r n_i})-2\min\{n_1,\ldots,n_r\}+1.\]
We give a geometric argument for the following improved bound:

\begin{thm}[\cref{cor:boundProductProjectiveSpacesGeneral}]
  For all $n_1,\ldots,n_r \geq 1$, we have
  \[\injDim{\P^{n_1} \times \ldots \times \P^{n_r}} \geq 2({\textstyle \sum_{i=1}^r n_i}) - \min\{n_1, \ldots, n_r\}.\]
\end{thm}

Moreover, we develop techniques for producing explicit injective morphisms into twice-dimensional projective spaces, see \cref{prop:tangentialInjections}, \cref{prop:injectP1P1}, \cref{thm:injectP1Pn} and \cref{cor:P1P1Pm}. The following summarizes the current knowledge on small injection dimensions for products of projective spaces:

\begin{thm}
  In the following cases, $\injDim{\P^{n_1} \times \ldots \times \P^{n_r}, \O(d_1,\ldots,d_r)} \leq 2 (\sum_{i=1}^r n_i)$ holds:
  \begin{itemize}[label={$\scriptstyle \blacktriangleright$}]
    \item $r = 1$, \hspace{1em} (previously known \cite[Proposition~5.2.2]{Duf08})
    \item $r = 2$, $n_1 = 1$, $d_2=1$,
    \item $r = 2$, $n_1 = n_2 = 2$, $d_1 = 1$, $d_2 = 2$,
    \item $r = 3$, $n_1 = n_2 = 1$, $d_1 = d_2 = d_3 = 1$.
  \end{itemize}
  Moreover, $\injDim{\P^{n_1} \times \ldots \times \P^{n_r}, \O(d_1,\ldots,d_r)} = 2 (\sum_{i=1}^r n_i) - 1$ holds for the cases
  \begin{itemize}[label={$\scriptstyle \blacktriangleright$}]
    \item $r = 2$, $d_1 = d_2 = 1$, \hspace{1em} (known as folklore, e.g.~\cite[Example~5.1.2.2]{Lan12})
    \item $r = 2$, $n_1 = n_2 = 1$, $d_1 = 1$, $d_2 = 2$.
  \end{itemize}
\end{thm}

\vspace{1em}

\noindent \textsc{Acknowledgments.} The author would like to thank Mateusz Michałek and Bernd Sturmfels for their guidance in this project, as well as Azeem Khadam for helpful in-depth discussions on local cohomology. 
\section{Secant avoidance and separating invariants}
In this section, we start out by establishing basic notions and gathering general observations on injective morphisms from arbitrary projective varieties to projective spaces and their relation to secant loci. We then highlight close interactions between injection dimensions and the theory of separating invariants. Finally, we relate the case of products of projective spaces to the identifiability of decomposable partially symmetric tensors under linear quotient operations. 

\subsection{General observations}

First, we fix some conventions for the entire article: Throughout, we work over the base field $\C$ and consider complex varieties, not assumed to be irreducible in general. %
For a finite-dimensional vector space $V$, we denote by $\C[V^*] := \Sym^\bullet V^*$ the graded ring of polynomial functions on $V$ and by $\P(V) := \Proj\: \C[V^*]$ the projective space parameterizing one-dimensional subspaces of $V$. For $v \in V \setminus \{0\}$, the corresponding point in $\P(V)$ is denoted $[v]$. The term subvariety (or subscheme, point etc.)\ refers to a \emph{closed} subvariety (subscheme, point etc.), unless mentioned otherwise.

We recall that a choice of global sections $f_0, \ldots, f_s \in H^0(X,\Ell)$ of a line bundle $\Ell$ on a variety $X$ determines a rational map to a projective space 
$X \dashrightarrow \P^s$.
In a coordinate-free manner, it 
is
the composition of the natural evaluation $\varphi_\Ell \colon X \dashrightarrow \P(H^0(X,\Ell)^*)$ with the projection $\P(H^0(X,\Ell)^*) \dashrightarrow \P(V^*)$, where 
$V$ is the subspace of $H^0(X,\Ell)$
spanned by $f_0, \ldots, f_s$. Conversely, for any non-zero subspace $V \subseteq H^0(X,\Ell)$, the composition $X \dashrightarrow \P(H^0(X,\Ell)^*) \dashrightarrow \P(V^*)$ is a well-defined rational map, which we denote by $\varphi_V$. In particular, $\varphi_\Ell = \varphi_{H^0(X, \Ell)}$ in our notations.

\begin{definition}[Injection dimension]
  Let $X$ be a projective variety and $\Ell$ a line bundle on $X$. The \defstyle{injection dimension of $X$ with respect to $\Ell$}, denoted $\injDim{X,\Ell}$, is defined as the smallest dimension of a projective space into which $X$ can be injected by global sections of $\Ell$. Formally, 
  \begin{align*}
    \injDim{X,\Ell} := \inf\big\{\dim V-1 \ {}\mid{} \ &V \subseteq H^0(X,\Ell) \text{ non-zero subspace such that } \\ %
    &\varphi_V \colon X \dashrightarrow \P(V^*) \text{ is an injective morphism}\big\}.
  \end{align*}
  We define the \defstyle{injection dimension of $X$} as
    \[\injDim{X} := \min\{\injDim{X,\Ell} \mid \Ell \text{ line bundle on } X\} = \min\{s \in \N \mid \text{there exists } X \injTo \P^s\}.\]
\end{definition}

Except for the case of projective spaces, the injection dimension is strictly larger than the dimension of the variety, as we note as an easy consequence of Zariski's Main Theorem:

\begin{lemma} \label{lem:dimPlusOne}
  Let $X$ be a projective variety not isomorphic to a projective space. 
  Then $\injDim{X} \geq \dim X + 1$.
\end{lemma}

\begin{proof}
  Let $n := \dim X$ and assume that there is an injective morphism $\varphi \colon X \injTo \P^n$. Since $\varphi$ is proper and finite, the restriction of $\varphi$ to any $n$-dimensional irreducible component of $X$ has an $n$-dimensional image and is therefore surjective. By injectivity of $\varphi$, $X$ must be irreducible and $\varphi \colon X \to \P^n$ is bijective. Being a finite surjective morphism of degree~$1$, the morphism $\varphi$ is birational. Then normality of projective space implies that $\varphi$ is an isomorphism by Zariski's Main Theorem (as e.g.\ in \cite[Exercise~29.6.D]{Vak17}).
\end{proof}

\begin{remark} \label{rem:dimPlusOne}
  More generally, the proof of \cref{lem:dimPlusOne} shows that the image of an injection $\varphi \colon X \injTo \P^n$ is a normal variety if and only if $X$ is normal and $\varphi$ is an isomorphism.
\end{remark}

If $\Ell$ is a very ample line bundle, i.e., if $\varphi_\Ell \colon X \dashrightarrow \P(H^0(X,\Ell)^*)$ is a closed embedding, then $\injDim{X, \Ell} \leq h^0(X,\Ell)-1 < \infty$. Conversely, if $\injDim{X, \Ell} < \infty$, then $\Ell$ is a globally generated ample line bundle, since $\Ell$ is the pullback of $\O_{\P(H^0(X,\Ell)^*)}(1)$ under the injective (hence finite) morphism $\varphi_\Ell$. Therefore, we have the implications
  \begin{equation} \label{eq:lineBundleConditions}
    \Ell \text{ very ample} \quad \Rightarrow \quad \injDim{X, \Ell} < \infty \quad \Rightarrow \quad \Ell \text{ ample and globally generated}. 
  \end{equation}
Note that the argument for the second implication also shows that non-projective complete varieties cannot be injected to projective spaces due to the lack of ample line bundles. In other words, a complete variety admits an injective morphism to a projective space if and only if it admits an embedding into a projective space.

The reverse implications of \cref{eq:lineBundleConditions} are not true, as the following examples show.

\begin{ex} \label{ex:injNotVeryAmple}
  Consider the weighted projective space $X = \P(1,6,10,15)$. The morphism 
    \[\P(1,6,10,15) \to \P^4, \qquad
      [x_0:x_1:x_2:x_3] \mapsto [x_0^{30} : x_0^{24}x_1 : x_0^{20}x_2 + x_1^5 : x_0^{15}x_3 + x_2^3 : x_3^2]\]
  is defined by global sections of $\Ell = \O(30)$ and it is injective, as we will confirm in \cref{thm:weightedProjectiveConstr} below. Hence, $\injDim{X,\Ell} = 4$. On the other hand, we observe that $\Ell$ is not very ample: The polarized toric variety $(\P(1,6,10,15), \Ell)$ corresponds to the lattice polytope $P := \conv(0,5e_1,3e_2,2e_3) \subseteq \R^3$. The semigroup $S := \N(P \cap \Z^3 - 5e_1)$ is not saturated in $\Z^3$, because 
    \[(-6,2,1) = \frac{1}{2}((-2,1,0)+(-5,3,0)+(-5,0,2)) \in \big(\frac{1}{2} S \cap \Z^3\big) \setminus S.\]
  By \cite[Proposition~6.1.10]{CLS11}, this shows that $\Ell$ is not very ample. In fact, we discuss in \cref{ex:minNotVeryAmple} that 
  $\injDim{X,\Ell^{\otimes k}} \geq 6$ for all $k \geq 2$, highlighting that to attain the smallest possible injection dimension, one cannot restrict to very ample line bundles only.
\end{ex}

\begin{ex}
  Let $X$ be an elliptic curve and $p \in X$. Then $\Ell = \O_X(2p)$ determines a double cover $\varphi_\Ell \colon X \to \P(H^0(X,\Ell)^*) \cong \P^1$. The non-injectivity of this morphism implies $\injDim{X, \Ell} = \infty$. On the other hand, $\Ell$ is globally generated and ample.
\end{ex}

Injection dimensions are closely tied to the behaviour of secant loci, as we point out next. In fact, this is an instance of the relation between higher secant loci of varieties and the study of $k$-regular maps, see e.g.\ \cite{BJJM19}. In that context, injective maps to low-dimensional ambient spaces appear under the name “$2$-regular maps” and are an important first case of interest.

\begin{definition}
  Let $Y$ be a subvariety of $\P(V)$, where $V$ is a finite-dimensional vector space. The \defstyle{secant locus} of $Y$ in $\P(V)$, denoted $\sigma_2^{\circ}(Y)$ is the set
  \[\sigma_2^{\circ}(Y) := \bigcup_{p,q \in Y} \langle p,q \rangle \subseteq \P(V),\]
  where $\langle p, q \rangle \subseteq \P(V)$ denotes the linear subspace spanned by the points $p$ and $q$. Its closure in $\P(V)$ is the \defstyle{secant variety} of $Y \subseteq \P^m$ and is denoted $\sigma_2(Y)$.
\end{definition}

Injection dimensions have a straightforward reinterpretation in terms of the smallest codimension of a linear space avoiding a secant locus, based on the following classical observation:

\begin{lemma} \label{lem:secantAvoidance}
  Let $W \subseteq V$ be finite-dimensional vector spaces, let $Y \subseteq \P (V)$ be a subvariety and consider the linear space $L := \P(W)$. 
  The rational map $\pi \colon \P(V) \dashrightarrow \P(V/W)$ (i.e., the projection from $L$) restricts to an injective morphism $\restr{\pi}{Y} \colon Y \injTo \P(V/W)$ if and only if $L \cap \sigma_2^{\circ}(Y) = \emptyset$.
\end{lemma}

\begin{proof}
  The projection from $L$ is a well-defined morphism on $Y$ if and only if $Y \cap L = \emptyset$. Let $y_1 \neq y_2 \in \P(V) \setminus L$. Then $\pi(y_1) = \pi(y_2)$ if and only if there are representatives $z_1, z_2 \in V \setminus \{0\}$ with $y_i = [z_i] \in \P(V)$ such that $z_1-z_2 \in W$. But 
  \[\{[z_1 - z_2] \in \P(V) \mid z_i \in V \setminus \{0\}, \ [z_i]=y_i\} = \langle y_1, y_2 \rangle.\]
  In particular, $\pi$ is well-defined and injective on $Y$ if and only if $\P(W) \cap \sigma_2^{\circ}(Y) = \emptyset$.
\end{proof}

\begin{prop} \label{prop:secantAvoidance}
  Let $X$ be a projective variety and let $\Ell$ be a line bundle on $X$ with $\injDim{X,\Ell} < \infty$. Let $Y$ be the image of the morphism $\varphi_\Ell \colon X \injTo \P(H^0(X,\Ell)^*)$. Then
    \[\injDim{X,\Ell} = \min\{\codim L - 1 \mid L \subset \P(H^0(X,\Ell)^*) \text{ linear subspace with } L \cap \sigma_2^{\circ}(Y) = \emptyset\}.\]
\end{prop}

\begin{proof}
  Let $V \subseteq H^0(X,\Ell)$ be a subspace and let $W := \ker(H^0(X,\Ell)^* \twoheadrightarrow V^*)$. %
  The projection $\pi \colon \P(H^0(X,\Ell)^*) \dashrightarrow \P(H^0(X,\Ell)^*/W) \cong \P(V^*)$ is the rational map $\varphi_V$ and, by \cref{lem:secantAvoidance}, it is an injective morphism on $Y$ if and only if $W \cap \sigma_2^{\circ}(Y) = \emptyset$. With the observation that $\dim V = \codim W$, this proves the claim.
\end{proof}

In particular, this gives the following folklore result which -- contrary to the study of closed embeddings -- does not require smoothness of the variety.

\begin{cor} \label{cor:TwiceDimPlusOne}
  Let $X$ be a projective variety. For any line bundle $\Ell$ with $\injDim{X,\Ell} < \infty$, we have $\injDim{X, \Ell} \leq 2 \dim X + 1$.
\end{cor}

\begin{proof}
  Note that $Y := \varphi_\Ell(X)$ %
  has the same dimension as $X$, since $\varphi_\Ell$ is injective. Since $\dim \sigma_2(Y) \leq 2 \dim Y + 1 = 2\dim X + 1$, a general linear subspace of $\P(H^0(X, \Ell)^*)$ of codimension $2 \dim X + 2$ does not meet $\sigma_2(Y)$, and in particular it avoids the secant locus of $Y$. The bound then follows from \cref{prop:secantAvoidance}.
\end{proof}

A brave generalization of \cref{qu:injectIntoTwiceDim} would be to ask whether $\injDim{X,\Ell} \leq 2\dim X$ holds for every very ample line bundle $\Ell$. This is not the case: As shown in \cite{Pie81}, a general canonical curve $X \subseteq \P^3$ of genus~$4$ does not admit a cuspidal projection to $\P^2$. By \cref{lem:secantAvoidance}, this means that $\injDim{X,\omega_X} = 3 > 2 \dim X$. On the other hand, it is conjectured in \cite[Conjecture~4.9]{DJ16} that for products of projective spaces, one does have $\injDim{\P^{n_1} \times \ldots \times \P^{n_r}, \O(d_1,\ldots,d_r)} \leq 2 (\sum_{i=1}^r n_i)$ for all $d_1,\ldots,d_r > 0$. A very simple case illustrating \cref{prop:secantAvoidance} is the following example.

\begin{ex}
  Let $X = \P^1 \times \P^1 \times \P^1$ and consider the very ample line bundle $\Ell = \O(1,1,1)$, corresponding to the Segre embedding
  \begin{align*}
    \varphi_\Ell \colon \P^1 \times \P^1 \times \P^1 = X &\hookrightarrow \P(H^0(X, \Ell)^*) \cong \P^7, \\
    [x_0:x_1] \times [y_0:y_1] \times [z_0:z_1] &\mapsto [x_i y_j z_k \mid i,j,k \in \{0,1\}]
  \end{align*}
  Here, the secant variety of $\varphi(X)$ fills the entire $7$-dimensional ambient space, but the secant locus $\sigma_2^{\circ}(\varphi(X))$ does not. For example, one can check that
    \[p := [(x_0 y_0 z_1)^* + (x_0 y_1 z_0)^* + (x_1 y_0 z_0)^*] \in \P(H^0(X, \Ell)^*) \setminus \sigma_2^{\circ}(\varphi_\Ell(X)).\]
  Hence, $\injDim{X,\Ell} \leq 6$ by \cref{prop:secantAvoidance} and an injection $\P^1 \times \P^1 \times \P^1 \injTo \P^6$ is obtained by projecting from the point $p$. Explicitly,
  \begin{align*}
    \P^1 \times \P^1 \times \P^1 &\injTo{} \P^6, \\
    \text{\scriptsize $[x_0:x_1] \times [y_0:y_1] \times [z_0:z_1]$} &{}\mapsto{} \text{\scriptsize $[x_0 y_0 z_0 : x_0 y_0 z_1 - x_0 y_1 z_0 : x_0 y_0 z_1 - x_1 y_0 z_0 : x_0 y_1 z_1 : x_1 y_0 z_1 : x_1 y_1 z_1]$.}
  \end{align*}
  In fact, one can check that $\P^7 \setminus \sigma_2^{\circ}(\varphi(X))$ does not contain any line, so by \cref{prop:secantAvoidance} there cannot exist an injection of $\P^1 \times \P^1 \times \P^1$ into $\P^5$ given by multilinear forms, showing $\injDim{\P^1 \times \P^1 \times \P^1, \O(1,1,1)} = 6$. We generalize this example in \cref{cor:P1P1Pm}, constructing an injective morphism $\P^1 \times \P^1 \times \P^m \injTo \P^{2(m+2)}$ for all $m \geq 1$.
\end{ex}

By \cref{prop:secantAvoidance}, the injection dimension $\injDim{X, \Ell}$ is determined by the largest-dimensional linear space avoiding the secant locus of $\varphi_\Ell(X)$. A natural question is whether “largest-dimensional” can be replaced by “maximal with respect to inclusion”:

\begin{question} \label{qu:Naive}
  Let $X$ be a projective variety, $\Ell$ a line bundle on $X$ and $V \subseteq H^0(X, \Ell)$ a subspace such that $\varphi_V$ is a closed embedding. 
  Does there exist a subspace $W \subseteq V$ of dimension $\injDim{X, \Ell}+1$ such that $\varphi_W$ is an injective morphism?
\end{question}

This question was raised in \cite{DJ16} and it was observed that a positive answer to it would show $\injDim{X,\Ell} \leq 2 \dim X$ whenever $X$ is smooth and $\Ell$ is a line bundle with $\sigma_2^{\circ}(\varphi_\Ell(X)) \neq \sigma_2(\varphi_\Ell(X))$. However, the following example gives a negative answer to \cref{qu:Naive}.

\begin{ex}
  Let $X = \P^1$, $\Ell = \O_{\P^1}(5)$. The subspace $V \subseteq H^0(X, \Ell)$ spanned by $f_0 := x_0^5$, $f_1 := x_0^4 x_1 + x_0^3 x_1^2$, $f_2 := x_0^2 x_1^3 + x_0 x_1^4$ and $f_3 := x_1^5$ defines a closed embedding
    \[\varphi_V \colon \P^1 = X \hookrightarrow \P(V^*) \cong \P^3, \qquad [x_0:x_1] \mapsto [f_0:f_1:f_2:f_3],\]
  describing a rational quintic space curve $C \subseteq \P^3$. %
  One can algorithmically confirm that every point in $\P^3$ lies on a secant line of $C$, i.e., $\sigma_2^{\circ}(C) = \P^3$. By \cref{lem:secantAvoidance}, this implies that for any $W \subsetneq V$, the projection $\P(V^*) \dashrightarrow \P(W^*)$ cannot be injective on $C$, hence $\varphi_{W}$ is not injective. Geometrically, this means that $C$ does not admit a cuspidal projection. On the other hand, $\injDim{\P^1,\O_{\P^1}(5)} = 2$ because of the injective morphism
    \[\P^1 \injTo \P^2, \qquad [x_0:x_1] \mapsto [x_0^5:x_0^4 x_1:x_1^5].\]
\end{ex}

\medskip
Throughout, the following elementary connection between injections in projective and affine settings will repeatedly come up:

\begin{lemma} \label{lem:injectiveOnAffineCones}
  Let $S = \bigoplus_{d \geq 0} S_d$ be a finitely generated graded algebra over $S_0 = \C$
  and let $V \subseteq S_1$ be a subspace. The rational map $\varphi_V \colon \Proj S \dashrightarrow \P(V^*)$ is an injective morphism if and only if the morphism of affine cones $\widehat{\varphi}_V \colon \Spec S \to V^*$ is injective.
\end{lemma}

\begin{proof}
  The graded ring homomorphism $\psi \colon \Sym^\bullet V \hookrightarrow \Sym^\bullet S_1 \to S$ induces the rational map $\varphi_V \colon \Proj S \dashrightarrow \Proj \, \Sym^\bullet V = \P(V^*)$ as well as the morphism between the affine cones $\widehat{\varphi}_V \colon \Spec S \to \Spec \, \Sym^\bullet V = V^*$. Note that $\varphi_V$ is a morphism if and only if $\widehat{\varphi}_V^{-1}(0)$ consists only of the closed point $o \in \Spec S$ corresponding to the maximal ideal $S_{\geq 1}$. In that case, we have the commutative diagram
  \[\begin{tikzcd}
  \Spec S \setminus \{o\}
  \arrow{r}[swap]{\restr{\widehat{\varphi}_V}{\Spec S \setminus \{o\}}}
  \arrow{d}{/\C^*}
  &[5em] V^* \setminus \{0\}
  \arrow{d}{/\C^*} \\
  \Proj S \arrow{r}{\varphi_V}
  &\P(V^*),
  \end{tikzcd}\]
  where the vertical morphisms are geometric quotients for the $\C^*$-actions induced by the gradings of $S$ and $\Sym^\bullet V$. Since $\psi$ is degree-preserving, $\widehat{\varphi}_V$ is $\C^*$-equivariant, hence $\varphi_V \colon \Proj S \dashrightarrow \P(V^*)$ is an injective morphism if and only if $\widehat{\varphi}_V$ is injective.
\end{proof}

\subsection{Graded separating invariants}

Classical Invariant Theory revolves around the problem of describing generators (and their relations) of invariant rings for group actions on vector spaces or, more generally, on varieties. However, generating sets of invariant rings tend to be very large (possibly infinite) and hard to explicitly construct. The study of \emph{rational invariants}, i.e., generators for quotient fields of invariant rings, is often simpler to carry out \cite{CTS07, HK07}, but in some applications describing only the generic behavior of the group action can be insufficient. An intermediate approach between these two extremes is the more recent field of study of \emph{separating invariants} \cite[§2.4]{DK15}, \cite{Kem09}, which maintain the full geometric information about orbit separation while remedying many of the complications of complete generating sets of invariants \cite{Dom07,DKW08,NS09}. Separating invariants are of major importance for applications, as for example in the recent work \cite{CCC19}, see \cite[§5]{DK15} for an overview of possible application.

Here, we highlight the close connection between separating invariants and injection dimensions of projective varieties. We focus on the most classical situation: separating invariants for linear actions of reductive algebraic groups on finite-dimensional vector spaces.

\begin{definition} \label{def:sepSet}
  A \defstyle{separating set of invariants} for a finite-dimensional representation $V$ of a group $G$ is a set of invariant polynomials $F \subseteq \C[V^*]^G$ such that for all points $v,w \in V$ the following equivalence holds:
  \[f(v) = f(w) \text{ for all } f \in F \quad \Leftrightarrow \quad f(v) = f(w) \text{ for all } f \in \C[V^*]^G.\]
\end{definition}

Equivalently, in the case of a reductive algebraic group $G$, a finite set of invariants $F = \{f_1,\ldots,f_s\} \subseteq \C[V^*]^G$ is separating if and only if the morphism $\gitFrac{V}{G} = \Spec \C[V^*]^G \to \A^s$ given by $(f_1,\ldots,f_s)$ is injective. This means that in the \emph{affine setting}, there is an immediate translation between injective morphism to affine spaces and separating sets of invariants, whenever the coordinate ring of an affine variety has a description as an invariant ring -- the difference being rather a change of language.

In this article, we look at the \emph{projective setting}: We study injective morphism from projective varieties to projective spaces. Here, the corresponding translation to the world of separating invariants is more subtle and we dedicate the remainder of this section to carefully working it out in detail.

\smallskip
Often, small separating sets of invariants are obtained in two steps: (1) identify a large separating set, (2) form a smaller separating set by taking suitable linear combinations. This is closely related to injections of projective varieties in the situation that the separating set in (1) consists of homogeneous polynomials satisfying homogeneous relations.

\begin{definition} \label{def:gradedSepSet}
  Let $V$ be a finite-dimensional representation of a group $G$.
  We call a finite separating set of invariants $F = \{f_1,\ldots, f_s\} \subseteq \C[V^*]^G$ \defstyle{graded} if each $f_i \in \C[V^*]^G$ is a homogeneous polynomial and its ideal of relations
  \[\ker\big(\C[z_1,\ldots,z_s] \to \C[V^*]^G, \quad z_i \mapsto f_i\big) \subseteq \C[z_1,\ldots,z_s]\]
  is homogeneous.
\end{definition}

Equivalently, a finite separating set $F = \{f_1,\ldots, f_s\}$ of homogeneous polynomials is graded if and only if the separating algebra $S := \C[F]$ can be given a grading $S = \bigoplus_{d \geq 0} S_d$ with $F \subseteq S_1$. We want to emphasize that this grading, induced by the homomorphism $\C[z_1,\ldots,z_s] \twoheadrightarrow \C[F]$, does typically \emph{not} agree with the natural grading of $\C[F]$ as a graded subalgebra of the polynomial ring $\C[V^*]$. See \cref{ex:gradingsOfSepSets} below.

For convenience, we formulated \cref{def:gradedSepSet} for finite separating sets, but note that this is not a restriction: A set $F \subseteq \C[V^*]^G$ spanning a finite-dimensional vector space $\langle F \rangle \subseteq \C[V^*]^G$ is a separating set if and only if a basis of $\langle F \rangle$ is.

\begin{ex} \label{ex:gradingsOfSepSets}
  The action of $\Z_6 = \{\xi \in \C^* \mid \xi^6 = 1\}$ on $V = \C^4$ given by
  \[\Z_6 \to \GL(5,\C), \qquad \xi \mapsto \diag(\xi^2, \xi^2, \xi^3, \xi^3).\]
  gives an invariant ring $\C[V^*]^{\Z_6}$ generated by the $7$~invariant homogeneous polynomials
  \[F := \{f_i := x_1^i x_2^{3-i}, \ g_j := x_3^j x_4^{2-j} \mid i \in \{0,1,2,3\}, \ j \in \{0,1,2\}\}.\]
  Their ideal of relations is homogeneous, generated by four quadratic binomials, so $F$ is a graded separating set of invariants. Note that there are two different gradings on $\C[F]$: With respect to the grading induced from $\C[V^*]$,  we have $\deg f_i = 3$ and $\deg g_j = 2$. On the other hand, with respect to the grading induced by $\C[z_1,\ldots,z_7] \twoheadrightarrow \C[F]$,  every element of $F$ is homogeneous of degree~$1$. A separating set of smallest cardinality obtained by linear combinations from elements in $F$ is $E := \{f_0,f_1,f_2+f_3, g_0, g_1, g_2\}$, see \cref{ex:wpsBoundNotSharp}.
\end{ex}

\begin{prop} \label{prop:connectionToSepInvSubsystem}
  Let $V$ be a finite-dimensional representation of a reductive algebraic group $G$. Let $F = \{f_1,\ldots,f_m\} \subseteq \C[V^*]^G$ be a graded separating set with ideal of relations $\mathfrak a_F := \ker(\C[z_1,\ldots,z_m] \twoheadrightarrow \C[F])$. Let $s \leq m-1$ be minimal such that the projective variety $V(\mathfrak a_F) \subseteq \P^{m-1}$ can be injected to $\P^s$ by a linear projection $\P^{m-1} \dashrightarrow \P^s$. Then 
    \[s = \min\{|E|-1 \mid E \subseteq \langle F \rangle \text{ separating set}\}.\]
  In particular, a lower bound for the size of a separating set obtained by linear combinations of $f_1,\ldots,f_m$ is given by $1+\injDim{\Proj \C[F], \O(1)}$, where we consider $\C[F]$ with the grading induced by $\C[z_1,\ldots,z_m] \twoheadrightarrow \C[F]$.
\end{prop}

\begin{proof}
  Replacing $F$ with a linearly independent subset, we reduce to the case that $F$ is a basis for $\langle F \rangle \subseteq \C[V^*]^G$. Let $E = \{g_1,\ldots,g_k\} \subseteq \langle F \rangle$ be a finite set of linear combinations of $f_1,\ldots,f_m$ with $\dim\, \langle E \rangle = k$. Consider the morphism
  \[\psi \colon \gitFrac{V}{G} = \Spec \C[V^*]^G \to \Spec \C[F] \to \Spec \C[E] \xhookrightarrow{(g_1,\ldots,g_k)} \A^k\]
  and note that $E$ is a separating set if and only if $\psi$ is injective on the set-theoretic image of the quotient morphism $V \to \gitFrac{V}{G}$. For reductive groups, the latter quotient morphism is surjective, so $E$ is separating if and only if $\psi$ is injective. 
  
  Since $F$ is a separating set, this means in particular that $\Spec \C[V^*]^G \to \Spec \C[F]$ is injective.  By \cite[Theorem~2.4.6]{DK15}, the assumption that $F$ consists of homogeneous polynomials (with respect to the grading of $\C[V^*]$) implies that this morphism is the normalization of $\Spec \C[F]$ and hence also surjective. In particular, $\psi$ is injective if and only if $\widehat{\varphi}_{\langle E \rangle} \colon \Spec \C[F] \xrightarrowinline{(g_1,\ldots,g_k)} \A^k$ is injective.
  
  Considering the grading $\C[F] = \bigoplus_{d \geq 0} S_d$ induced by $\C[z_1,\ldots,z_m] \twoheadrightarrow \C[F]$, we have $\Proj \C[F] \cong V(\mathfrak a_F) \subseteq \P^{m-1}$ and $E \subseteq S_1$. Then the above $\widehat{\varphi}_{\langle E \rangle}$ is the morphism of affine cones over the rational map $\varphi_{\langle E \rangle} \colon \Proj \C[F] \dashrightarrow \P(\langle E \rangle^*) = \P^{k-1}$ given by $g_1,\ldots,g_k \in S_1 \subseteq H^0(\Proj \C[F], \O(1))$. By \cref{lem:injectiveOnAffineCones}, injectivity of $\widehat{\varphi}_{\langle E \rangle}$ is equivalent to $\varphi_{\langle E \rangle}$ being an injective morphism. To sum up, a linearly independent set $E \subseteq \langle F \rangle$ is separating if and only if 
    \[\varphi_{\langle E \rangle} \colon \Proj \C[F] \cong V(\mathfrak a_F) \subseteq \P^{m-1} = \P(\langle F \rangle^*) \dashrightarrow \P(\langle E \rangle^*)\]
  is an injective morphism, proving the claim.
\end{proof}

A classical fact about separating invariants \cite[Proposition~5.1.1]{Duf08} is that there always exists a separating set of size $2\dim \C[V^*]^G+1$. Note that in the presence of a graded separating set $F$, this bound can be improved by one, combining \cref{prop:connectionToSepInvSubsystem} and \cref{cor:TwiceDimPlusOne}. An example of a representation not admitting a graded separating set is $\Z_3 \to \GL(2,\C)$, $\xi \mapsto \diag(\xi,\xi^2)$: Its invariant ring $\C[x_1,x_2]^{\Z_3} = \C[x_1^3, x_1 x_2, x_2^3]$ contains no homogeneous polynomials separating orbits and satisfying homogeneous relations.

\begin{prop} \label{prop:connectionToSepInvComplete}
  Let $V$ be a finite-dimensional representation of a reductive algebraic group $G$. Assume that the invariant ring $S := \C[V^*]^G$ can be given a grading $S = \bigoplus_{d \geq 0} S_d$ with $S_0 = \C$ such that $S_1$ is a separating set. Then, with respect to this grading,
  \[\injDim{\Proj S,\O(1)} = \min\{|F| - 1 \mid F \subseteq S_1 \text{ finite separating set}\}.\]
\end{prop}

As before, we remark that the grading of $S = \C[V^*]^G$ in \cref{prop:connectionToSepInvComplete} need not agree with the grading induced from the polynomial ring $\C[V^*]$. Moreover, note that  \cref{prop:connectionToSepInvComplete} does not assume $S_1$ to be spanned by homogeneous polynomials with respect to the $\C[V^*]$-grading (in which case \cref{prop:connectionToSepInvComplete} is a consequence of \cref{prop:connectionToSepInvSubsystem}).

\begin{proof}[Proof of \cref{prop:connectionToSepInvComplete}]
  Since $G$ is reductive, the invariant ring $S = \C[V^*]^G$ is a finitely generated $\C$-algebra and, in particular, $S_1$ is a finite-dimensional vector space.
  Moreover, the reductivity of $G$ implies that a subset $F \subseteq S_1$ is separating if and only if the morphism $\widehat{\varphi}_{\langle F \rangle} \colon \Spec S \to \langle F \rangle^*$ is injective. By \cref{lem:injectiveOnAffineCones}, this is equivalent to $\varphi_{\langle F \rangle} \colon \Proj S \dashrightarrow \P(\langle F \rangle^*)$ being an injective morphism.
  
  By assumption, this is the case for the separating set $F = S_1$, so in particular the rational map $\varphi_{S_1} \colon \Proj S \dashrightarrow \P(S_1^*)$ is a morphism. This means that the vanishing set of the homogeneous ideal $(S_1) \subseteq S$ in $\Proj S$ is empty, hence the coherent sheaf $\O_{\Proj S}(1)$ is a line bundle. Its global sections are $H^0(\Proj S, \O(1)) = S_1$, since $S = \C[V^*]^G$ is a normal ring, see \cite[Proposition~2.4.4]{DK15}. In particular, every rational map from $\Proj S$ to a projective space given by a linear subsystem of $|\O(1)|$ is of the form $\varphi_{\langle F \rangle}$ for some $F \subseteq S_1$. Then the previous observation proves the claim.
\end{proof}

\begin{remark}
  \cref{prop:connectionToSepInvComplete} and its proof generalize verbatim to the setting that the finite-dimensional representation $V$ of $G$ is replaced by the action of a reductive algebraic group $G$ on a normal irreducible affine variety $X = \Spec R$, with the obvious generalization of \cref{def:sepSet} to this case, as in \cite[Definition~2.4.1]{DK15}.
\end{remark}

\cref{prop:connectionToSepInvSubsystem} and \cref{prop:connectionToSepInvComplete} show that graded separating sets have an interpretation as injective morphisms of projective varieties to projective spaces by subsystems of a fixed line bundle. Conversely, one can often interpret injections of a given projective variety as separating sets with respect to a suitable invariant ring. We highlight this in the setting of normal toric varieties:

\begin{thm} \label{prop:toricInjAsSepSet}
  Let $\Ell$ be an ample line bundle on a normal projective toric variety $X$.
  There is a finite-dimensional representation $V$ of 
  a diagonalizable group $G$ and a grading $S = \bigoplus_{d \geq 0} S_d$ of the invariant ring $S := \C[V^*]^G$ such that
  \[\injDim{X,\Ell} = \inf\{|F| - 1 \mid F \subseteq S_1 \text{ finite separating set of invariants}\}.\]
\end{thm}

\begin{proof}
  Let $X = X_\Sigma$ be the normal toric variety associated to a fan of rational polyhedral cones $\Sigma$. 
  Its Cox ring $\Cox(X_\Sigma)$ is the polynomial ring $\C[x_\rho \mid \rho \in \Sigma(1)]$ and it is graded by the class group $\Cl(X)$, which is a finitely generated abelian group, see \cite[§5.2]{CLS11}. If $\Ell \in \Pic(X)$ corresponds to $\alpha \in \Cl(X)$ under the inclusion $\Pic(X) \hookrightarrow \Cl(X)$, then the $\alpha$-graded piece of the Cox ring is $\Cox(X_\Sigma)_\alpha = H^0(X, \Ell)$ by \cite[Proposition~5.4.1]{CLS11}.
  
  The grading of the Cox ring by the class group corresponds to a linear action of the character group $G_0 := \Hom_\Z(\Cl(X),\C^*)$ on the vector space $V := \C^{\Sigma(1)} = \Spec \Cox(X_\Sigma)$, and the graded pieces of $\Cox(X_\Sigma)$ are the eigenspaces for the induced action of $G_0$ on $\C[V^*] = \Cox(X_\Sigma)$. Since $\Cl(X)$ is a finitely generated abelian group, $G_0$ is a diagonalizable group (i.e., the product of a torus and a finite abelian group).
  
  Consider the subgroup $G := \{\xi \in G_0 \mid \xi(\alpha) = 1\}$, which is again diagonalizable. Then a homogeneous element $f \in \Cox(X_\Sigma)_\beta$ is invariant under the action of $G$ if and only if $\xi(\beta) = 1$ for all $\xi \in G$. By \cite[Exercise~3.2.10.(4)]{Spr98}, this is only the case when $\beta$ lies in the subgroup of $\Cl(X)$ generated by $\alpha$. This means
    \[\C[V^*]^G = \bigoplus_{d \in \Z} \Cox(X_\Sigma)_{d\alpha} \cong \bigoplus_{d \in \Z} H^0(X,\Ell^{\otimes d}) = \bigoplus_{d \in \N} H^0(X,\Ell^{\otimes d}),\]
  where the last equality follows from the fact that $H^0(X,\Ell^{\otimes d}) = 0$ for all $d < 0$, since $\Ell$ is ample. Defining $S_d := H^0(X,\Ell^{\otimes d})$, this gives a grading $S = \bigoplus_{d \geq 0} S_d$ on the invariant ring $S := \C[V^*]^G$ with $S_1 = H^0(X,\Ell)$. Then the claim follows from \cref{prop:connectionToSepInvComplete}, using that $X \cong \Proj S$. Note that $\injDim{X,\Ell} = \infty$ if and only if $S_1 = H^0(X,\Ell)$ is not a separating set.
\end{proof}

\begin{ex} \label{ex:wpsAsSepInv}
  We exemplify \cref{prop:toricInjAsSepSet} in the case of a weighted projective space $X = \P(q_0,\ldots,q_n)$. Its Cox ring is the polynomial ring $\C[x_0,\ldots,x_n]$ as a $\Z$-graded ring with $\deg(x_i) = q_i$. This grading corresponds to the action of $G_0 = \C^*$ on $V = \C^{n+1}$ given by $\C^* \to \GL(n+1, \C)$, $t \mapsto \diag(t^{q_0},\ldots,t^{q_n})$. Every ample line bundle on $X$ is of the form $\Ell \cong \O(k)$ with $k > 0$ divisible by all $q_i$. Its section ring $\bigoplus_{d \in \N} H^0(X,\Ell^{\otimes d})$ is the $k$-th Veronese subring $\bigoplus_{d \geq 0} \C[x_0,\ldots,x_n]_{dk}$, which is the invariant ring for the action of the subgroup $G := \{\xi \in \C^* \mid \xi^k = 1\} \subseteq G_0$ on $V$.
  
  Hence, $\injDim{\P(q_0,\ldots,q_n),\O(k)}+1$ is the smallest size of a separating set of invariants $F \subseteq \C[x_0,\ldots,x_n]^{\Z_k}$ for the representation 
    \[\Z_k \to \GL(n+1,\C), \quad \xi \mapsto \diag(\xi^{q_0},\ldots,\xi^{q_n}),\]
  such that the polynomials in $F$ are homogeneous of $(q_0,\ldots,q_n)$-weighted degree $k$.
\end{ex}

\subsection{Segre--Veronese varieties and partially symmetric tensors} \label{ssec:SegreVeronese}

A major source of examples in the constructive parts of this article are products of projective spaces. Every ample line bundle on $X = \P^{n_1} \times \ldots \times \P^{n_r}$ is very ample, so \eqref{eq:lineBundleConditions} makes clear which line bundles give rise to injections: For $(d_1,\ldots,d_r) \in \Z^r$ we have
\[\injDim{X,\O(d_1,\ldots,d_r)} < \infty \qquad \Leftrightarrow \qquad d_1,\ldots,d_r > 0.\]

For $\Ell = \O(d_1,\ldots,d_r)$ with $d_i \geq 1$, the vector space $H^0(X, \Ell)$ can be identified with the space of multihomogeneous forms of degree $\dVec = (d_1,\ldots,d_r)$:
\[H^0(X,\Ell) = \bigotimes_{i=1}^r \Sym^{d_i} (\C^{n_i+1})^*.\]
Then the closed embedding $\varphi_\Ell \colon X \hookrightarrow \P(H^0(X, \Ell)^*)$ is the natural morphism
\begin{align*}
\varphi_\Ell \colon \P(\C^{n_1+1}) \times \ldots \times \P(\C^{n_r+1})  &\hookrightarrow \P\left(\bigotimes_{i=1}^r \Sym^{d_i} \C^{n_i+1}\right), \\
[v_1] \times \ldots \times [v_r] &\mapsto [v_1^{d_1} v_2^{d_2} \ldots v_r^{d_r}].
\end{align*}
Its image $Y := \im(\varphi_\Ell)$ is a subvariety of $\P(\bigotimes_{i=1}^r \Sym^{d_i} \C^{n_i+1})$ called the \defstyle{Segre--Veronese variety} of type $(n_1,\ldots,n_r;d_1,\ldots,d_r)$. The space $\P(\bigotimes_{i=1}^r \Sym^{d_i} \C^{n_i+1})$ consists of partially symmetric tensors up to scaling, and the subvariety $Y$ consists of the \emph{decomposable} (or \emph{rank~1}) partially symmetric tensors. Its secant locus $\sigma_2^{\circ}(Y) \subseteq \P(\bigotimes_{i=1}^r \Sym^{d_i} \C^{n_i+1})$ is the set of partially symmetric tensors (up to scaling) of \emph{rank at most~2}. 
We refer the reader to \cite[§9]{MS19} for a brief introduction to varieties of tensors, their ranks and secant loci. For in-depth background on the theory of (partially symmetric) tensors and their importance in applications, see \cite{Lan12} .
In this language, \cref{prop:secantAvoidance} gives the following reinterpretation of $\injDim{\P^{n_1} \times \ldots \times \P^{n_r},\O(d_1, \ldots, d_r)}$:

\begin{cor}
  Let $X = \P^{n_1} \times \ldots \times \P^{n_r}$ and $\Ell := \O(d_1, \ldots, d_r)$ for $n_i, d_i \geq 1$. Then
  \begin{align*}
  \injDim{X,\Ell} = \min\big\{\codim L - 1 \ \mid \ &L \subseteq \bigotimes_{i=1}^r \Sym^{d_i} \C^{n_i+1} \text{ subspace not containing any} \\ &\text{non-zero partially symmetric tensor of rank $\leq 2$}\big\}.
  \end{align*}
\end{cor}

In other words, the search for a low-dimensional injection of $\P^{n_1} \times \ldots \times \P^{n_r}$ by polynomials of multidegree $(d_1,\ldots,d_r)$ is equivalent to finding a high-dimensional subspace $L \subseteq \bigotimes_{i=1}^r \Sym^{d_i} \C^{n_i+1}$ such that decomposable partially symmetric tensors stay identifiable under the quotient $\Sym^{d_i} \C^{n_i+1} \twoheadrightarrow \Sym^{d_i} \C^{n_i+1}/L$ (in the sense that any decomposable tensors can be uniquely reconstructed from its image under the quotient operation).

\medskip
By \cref{prop:toricInjAsSepSet}, this question also has a formulation in terms of separating invariants. We work this out carefully here, since an incorrect description in the literature gave rise to wrong lower bounds on injection dimensions \cite{DJ16}. We comment on this unfortunate flaw in the literature and its correction at the beginning of \cref{sec:lowerBounds}.

The Cox ring of $X = \prod_{i=1}^r \P^{n_i}$ is the polynomial ring $\C[V^*]$, where $V := \bigoplus_{i=1}^k \C^{n_i+1}$. Explicitly, denoting
the coordinates on $\C^{n_i+1}$ by $x_{i0}, x_{i1},\ldots,x_{in_i}$, this is
  \[\Cox(\P^{n_1} \times \ldots \times \P^{n_r}) = \C[x_{ij} \mid i \in \{1,\ldots,r\}, j \in \{0,1,\ldots,n_i\}],\]
equipped with a $\Z^r$-grading given by $\deg x_{ij} = e_i \in \Z^r$. This grading corresponds to the action of $G_0 = (\C^*)^r$ on $V := \bigoplus_{i=1}^r \C^{n_i+1}$ given by 
  \[(t_1, \ldots, t_r) \cdot (v_1,\ldots,v_r) := (t_1 v_1, \ldots, t_r v_r) \quad \text{for all } v = (v_1, \ldots, v_r) \in V.\]
Every ample line bundles on $X$ is of the form $\Ell \cong \O(d_1,\ldots,d_r)$ with $d_i > 0$. Its section ring $\bigoplus_{k \in \N} H^0(X,\Ell^{\otimes k})$ is the homogeneous coordinate ring of the Segre--Veronese variety of type $(n_1,\ldots,n_r; d_1,\ldots,d_r)$. It is the invariant ring for the action on $V$ of the subgroup 
  \[G := \{(t_1, \ldots, t_r) \in (\C^*)^r \mid t_1^{d_1} \cdots t_r^{d_r} = 1\} \subseteq G_0,\]
which is isomorphic to $(\C^*)^{r-1} \times \Z_k$, where $k := \gcd\{d_1,\ldots,d_r\}$. Explicitly, consider the action of $(\C^*)^{r-1} \times \Z_k$ on $V$ given by 
  \[((t_1,\ldots,t_{r-1}) \times \xi) \times e_{ij} \mapsto \begin{cases} t_i^{d_1 \cdots \widehat{d_i} \cdots d_r} \: \xi \: e_{ij} &\text{if } i \leq r-1 \\ (t_1 \cdots t_{r-1})^{-d_1 \cdots d_{r-1}} \: \xi \: e_{ij} &\text{if } i = r.\end{cases}\]
Its invariant ring is generated by all monomials of multidegree $(d_1,\ldots,d_r)$. A separating set of invariants consisting of $s$ linear combinations of these corresponds to an injection of $\P^{n_1} \times \ldots \times \P^{n_r} \injTo \P^{s-1}$ given by global sections of $\O(d_1,\ldots,d_r)$.

\section{Obstructions to low-dimensional injections} \label{sec:lowerBounds}
In this section, we provide lower bounds on injection dimensions due to topological obstructions: The first approach (\cref{cor:lowerBoundFromMorphism}) exploits that in projective spaces any two subvarieties of complementary dimension must intersect, while this is not necessarily the case for arbitrary projective varieties -- this discrepancy leads to lower bounds on injection dimensions, irrespective of the choice of a line bundle. We use this simple argument to improve previously known lower bounds on $\injDim{\P^{n_1} \times \ldots \times \P^{n_r}}$.

Secondly, a more sophisticated argument based on (dis-)connectedness properties for orbits of linear spaces under a finite group action bounds injection dimensions for line bundles which admit a root of some order (\cref{thm:lowerBound}) or are divisible in the class group (\cref{thm:lowerBoundReflexive}). We apply this to construct $n$-dimensional irreducible varieties of injection dimension $\geq 2n$ and comment on injection dimensions of weighted projective spaces.

Previous work from the perspective of separating invariants \cite{Duf08}, \cite{DJ14} established that $\injDim{\P^n, \O(d)} = 2n$ for all $d \geq 2$, whereas of course $\injDim{\P^n, \O(1)} = n$. There, the (slightly stronger) question of injecting the affine cones over Veronese varieties into affine spaces is studied with techniques from local cohomology in order to obtain lower bounds.

With a similar approach, the article \cite{DJ16} claims to prove for products of projective spaces $\P^{n_1} \times \ldots \times \P^{n_r}$ that the injection dimension with respect to a very ample line bundle $\O(d_1,\ldots,d_r)$ is bounded below by $2(\sum_{i=1}^r n_i)$, provided that not all $d_i$ are equal to $1$. Unfortunately, this is wrong -- we provide an explicit counterexample in \cref{prop:tangentialInjections} showing $\injDim{\P^1 \times \P^1, \O(1,2)} = 3$. The flaw in \cite{DJ16} is an incorrect description of a group action identifying the homogeneous coordinate ring of a Segre--Veronese variety with an invariant ring in the case that not all $d_i$ are equal. Correcting this with the description given in \cref{ssec:SegreVeronese}, a straightforward adaption of \cite[Proof of Theorem~5.4]{DJ16} shows $\injDim{\P^{n_1} \times \ldots \times \P^{n_r},\O(d_1,\ldots,d_r)} \geq 2 (\sum_{i=1}^r n_i)$ whenever $\gcd\{d_1,\ldots,d_r\} > 1$.

\cref{thm:lowerBound} and \cref{thm:lowerBoundReflexive} generalize these results from (products of) projective spaces to arbitrary projective varieties.

\subsection{Existence of fibrations}

We start out with an elementary observation giving lower bounds on injection dimensions:

\begin{lemma} \label{prop:boundByNonemptyIntersections}
  Let $X$ be a projective variety and let $Y,Z \subseteq X$ be disjoint closed subsets. Then
    \[\injDim{X} \geq \dim Y + \dim Z + 1.\]
\end{lemma}

\begin{proof}
  For any injection $\varphi \colon X \injTo \P^s$, we have $\varphi(Y) \cap \varphi(Z) = \varphi(Y \cap Z) = \emptyset$ and $\dim \varphi(Y) = \dim Y$, $\dim \varphi(Z) = \dim Z$. But two subvarieties of $\P^s$ can only be disjoint if their dimensions sum to at most $s-1$, hence $s \geq \dim Y + \dim Z + 1$.
\end{proof}

\begin{ex}
  Let $L \subseteq \P^n$ be a linear subspace of codimension~$2$ and consider the blow-up of $\P^n$ along $L$. Then the strict transforms of two distinct hyperplanes containing $L$ are disjoint effective divisors on $\Bl_L \P^n$. Hence, $\injDim{\Bl_L \P^n} \geq 2n-1$ by \cref{prop:boundByNonemptyIntersections}.
\end{ex}

A particular consequence of \cref{prop:boundByNonemptyIntersections} is that the existence of fibrations $X \to S$ is an obstruction to low-dimensional injections of $X$:

\begin{prop} \label{cor:lowerBoundFromMorphism}
  Let $X \to S$ be a surjective morphism of irreducible projective varieties with $\dim S \geq 1$. Then
    \[\injDim{X} \geq 2 \dim X - \dim S.\]
\end{prop}

\begin{proof}
  We can find disjoint irreducible subvarieties $Y_0,Z_0 \subseteq S$ with $\dim Y_0 + \dim Z_0 = \dim S - 1$. Let $Y,Z \subseteq X$ be the fibers of $X \to S$ over them. Note that $Y \cap Z = \emptyset$, so
  \cref{prop:boundByNonemptyIntersections} implies
  \begin{align*}
    \injDim{X} &\geq \dim Y + \dim Z + 1 \geq (\dim X - \dim S + \dim Y_0) + (\dim X - \dim S + \dim Z_0)+1 \\&\geq 2\dim X - \dim S. \qedhere
  \end{align*}
\end{proof}

\begin{ex}
  An $n$-dimensional projective bundle over a curve has injection dimension at least $2n - 1$.
\end{ex}

In \cite[Proposition~5.6]{DJ16}, the following general bound for injection dimensions of products of projective spaces was derived with techniques from local cohomology:
  \[\injDim{\P^{n_1} \times \ldots \times \P^{n_r}} \geq 2\big(\sum_{i=1}^r n_i\big) - 2\min\{n_1, \ldots, n_r\}+1.\]
Our basic geometric observations improve this bound as follows:

\begin{cor} \label{cor:boundProductProjectiveSpacesGeneral}
  For all $n_1, \ldots, n_r \geq 1$, we have
  \[\injDim{\P^{n_1} \times \ldots \times \P^{n_r}} \geq 2\big(\sum_{i=1}^r n_i\big) - \min\{n_1, \ldots, n_r\}.\]
\end{cor}

\begin{proof}
  This follows from applying \cref{cor:lowerBoundFromMorphism} to the projections $\P^{n_1} \times \ldots \times \P^{n_r} \to \P^{n_i}$ for $i=1,\ldots,r$.
\end{proof}

\subsection{Divisibility in the Picard/class group}

Our main lower bound for injection dimension with respect to fixed line bundles follows. 
It is inspired by and vastly generalizes previous work for (products of) projective spaces in \cite{DJ14,DJ16}.

\begin{thm} \label{thm:lowerBound}
  Let $X$ be a projective variety and let $\Ell$ be a line bundle on $X$. Then $\injDim{X,\Ell^{\otimes k}} \geq 2\dim X$ for all $k \geq 2$.
\end{thm}

\begin{proof}
  By restricting to a top-dimensional component, we may assume that $X$ is irreducible. Fix $k \geq 2$. We may assume that $\injDim{X,\Ell^{\otimes k}} < \infty$, since the claim is otherwise trivial. Then $\Ell$ is ample by \eqref{eq:lineBundleConditions}, so by \cite[Example~1.2.22]{Laz04}, its section ring $R := \bigoplus_{d \geq 0} H^0(X,\Ell^{\otimes d})$ is a finitely generated graded algebra over $R_0 = H^0(X,\O_X) = \C$, and we have $\Proj R \cong X$. The $\C$-algebra $S := R \otimes_\C R$ is then also finitely generated, and it inherits a grading $S = \bigoplus_{d \geq 0} S_d$ with graded pieces $S_d = \bigoplus_{i=0}^d R_i \otimes_\C R_{d-i}$.
  
  The Veronese subalgebra $R^\supind{k} := \bigoplus_{d \geq 0} R_{kd}$ is the section ring of the line bundle $\Ell^{\otimes k}$ and we have $\Proj R^\supind{k} \cong \Proj R \cong X$. Note that $R^\supind{k}$ is the invariant ring under the degree-preserving action of the cyclic group $\Z_k = \{\xi \in \C^* \mid \xi^k = 1\}$ on $R$ given by
  \[\Z_k \times R_d \to R_d, \qquad (\xi,f) \mapsto \xi^d f.\]

  Let $V \subseteq H^0(X,\Ell^{\otimes k}) = R_k$ be a non-zero subspace such that $\varphi_V \colon X \injTo \P(V^*)$ is an injective morphism. We aim to show that $\dim \P(V^*) \geq 2\dim X$.
  We consider the following commutative diagram:
  \begingroup \footnotesize
  \[\begin{tikzcd}[column sep=small]
  &
  \Spec R^\supind{k} \otimes_\C R^\supind{k}
  \arrow{dr}[swap]{\hat{\varphi}_V \times \hat{\varphi}_V}
  \arrow[dashed]{dl}
  & 
  &
  \Spec R \otimes_\C R
  \arrow[dashed]{dr}
  \arrow{dl}{\hat{\psi}}
  \arrow{ll}{/{(\Z_k \times \Z_k)}}[swap]{\pi}
  &
  \\ 
  \Proj R^\supind{k} \times \Proj R^\supind{k}
  \arrow{dr}{\varphi_V \times \varphi_V}
  & 
  & 
  V^* \times V^*
  \arrow[dashed]{dr} 
  \arrow[dashed]{dl}
  &
  & \Proj R \otimes_\C R
  \arrow{dl}[swap]{\psi}
  \\ 
  & 
  \P(V^*) \times \P(V^*)
  &
  &
  \P(V^* \times V^*).
  & 
  \end{tikzcd}\]
  \endgroup
  
  Injectivity of $\varphi_V$ means that the preimage of the diagonal in $\P(V^*) \times \P(V^*)$ under $\varphi_V \times \varphi_V$ is set-theoretically the diagonal in $X \times X$, i.e.,
  \[((\varphi_V \times \varphi_V)^{-1}\Delta_{\P(V^*) \times \P(V^*)})^{\text{red}} = \Delta_{\Proj R^\supind{k} \times \Proj R^\supind{k}}.\]
  On the level of affine cones, the morphism $\hat{\varphi}_V \colon \Spec R^\supind{k} \to V^*$ is injective by \cref{lem:injectiveOnAffineCones}, so this equality lifts to
    \[((\hat{\varphi}_V \times \hat{\varphi}_V)^{-1}\Delta_{V^* \times V^*})^{\text{red}} = \Delta_{\Spec R^\supind{k} \times \Spec R^\supind{k}}.\]
  The natural morphism $\pi \colon \Spec R \otimes_\C R \to \Spec R^\supind{k} \otimes_\C R^\supind{k}$ is the geometric quotient for the action of $\Z_k \times \Z_k$ on $\Spec S = \Spec R \times \Spec R$, so the preimage of the diagonal under $\hat{\psi} := (\hat{\varphi}_V \times \hat{\varphi}_V) \circ \pi$ is the $\Z_k \times \Z_k$-orbit of the diagonal in $\Spec R \times \Spec R$:
  \[(\hat{\psi}^{-1}\Delta_{V^* \times V^*})^{\text{red}} = (\Z_k \times \Z_k) \cdot \Delta_{\Spec R \times \Spec R}.\]
  
  Under projectivization of $V^* \times V^*$, the diagonal $\Delta_{V^* \times V^*}$ becomes a linear subspace $L \subseteq \P(V^* \times V^*)$ of dimension $\dim \P(V^*)$. Then the previous equality of sets becomes
  \[(\psi^{-1}L)^{\text{red}} = (\Z_k \times \Z_k) \cdot Y = \bigcup_{\xi \in \Z_k} (1\times \xi) \cdot Y,\]
  where $Y := V(f \otimes 1 - 1 \otimes f \mid f \in R) \subseteq \Proj R \otimes_\C R$.

  We claim that this is in fact a disjoint union, so that $\psi^{-1}(L)$ has $k \geq 2$ connected components. Then, by \cite[Theorem~3.3.3]{Laz04}, this disconnectedness forces
    \[\codim_{\P(V^* \times V^*)} L \geq \dim \im \psi.\]
  Note that $\codim_{\P(V^* \times V^*)} L = \dim \P(V^*)+1$. On the other hand, 
    \[\dim \im \psi = \dim \im \hat{\psi} -1 = \dim \im (\hat{\varphi}_V \times \hat{\varphi}_V) - 1 = \dim \im (\varphi_V \times \varphi_V) + 1 = 2 \dim X + 1,\]
  where the last equality holds by injectivity of $\varphi_V$. We conclude that $\P(V^*) \geq 2\dim X$.
  
  It remains to prove that $(1 \times \xi) \cdot Y$ and $(1 \times \xi') \cdot Y$ are disjoint for $\xi \neq \xi' \in \Z_k$. For this, we may assume that $\xi' = 1$. Note that $(1 \times \xi) \cdot Y = V(f \otimes 1 - 1 \otimes \xi^d f \mid f \in R_d, d \geq 0)$. Therefore,
    \[Y \cap (1 \times \xi) \cdot Y = Y \cap V(f \otimes 1,\: 1 \otimes f \;\mid\; f \in R_d,\: d \geq 0 \text{ with }\xi^d \neq 1).\]
  In particular, we have 
  \begin{equation} \label{eq:intersectionRelPrime}
    Y \cap (1 \times \xi) \cdot Y \subseteq V(R_d \otimes 1 + 1 \otimes R_d) \qquad \text{for all $d \geq 0$ with } (d,k)=1.
  \end{equation}
  
  For $d \gg 0$, the line bundle $\Ell^{\otimes d}$ is globally generated, so the vanishing locus of $R_d = H^0(X,\Ell^{\otimes d})$ inside $\Proj R \cong X$ is empty. This means that for $d\gg 0$, we have $\sqrt{(R_d)} \supseteq R_{\geq 1}$, which implies $\sqrt{(R_d \otimes 1 + 1 \otimes R_d)} \supseteq (R_{\geq 1} \otimes 1 + 1 \otimes R_{\geq 1}) = (R \otimes_\C R)_{\geq 1}$.
  Then \eqref{eq:intersectionRelPrime} shows $Y \cap (1 \times \xi) \cdot Y = \emptyset$, proving the claim.
\end{proof}

We use \cref{thm:lowerBound} to give an example indicating that we cannot do better than what we ask for in \cref{qu:injectIntoTwiceDim}, even as the dimension of the varieties increase:

\begin{ex} \label{ex:noSmallInjectionSingular}
  Let $(X,\Ell)$ be the polarized normal toric variety of dimension~$n \geq 3$ corresponding to the full-dimensional lattice polytope 
    \[P := \conv(0,e_1,\ldots,e_{n-1},e_1+e_2+\ldots+e_{n-1}+n e_n) \subseteq \R^n.\]
  Then $\Pic(X) \cong \Z$ is generated by $\Ell$, and the complete linear system $|\Ell|$ determines a non-injective finite morphism $\varphi_\Ell \colon X \to \P^n$. In particular, $\injDim{X,\Ell^{\otimes k}} = \infty$ for all $k \leq 1$. For $k \geq 2$, \cref{thm:lowerBound} shows that $\injDim{X,\Ell^{\otimes k}} \geq 2\dim X$. In particular, $\injDim{X} \geq 2\dim X$, i.e., $X$ cannot be injected to $\P^s$ for $s < 2 \dim X$. For smooth examples with the same property see \cref{ex:noSmallInjectionSmooth}.
\end{ex}

In the case of normal varieties, \cref{thm:lowerBound} can be sharpened for singular situations, replacing the assumption on divisibility in the Picard group by divisibility in the class group. Then we obtain the following bound:

\begin{thm} \label{thm:lowerBoundReflexive}
  Let $X$ be a normal projective variety, let $D$ be a Weil divisor on $X$, let $k \geq 2$ and assume that $kD$ is Cartier. Then
    \[\injDim{X,\O_X(kD)} \geq 2 \dim X - \delta,\]
  where 
    \[\delta := \min\{1+\dim\big(\bigcap_{q \nmid m} \Bs |mD|\big) \mid q \text{ prime power dividing } k\},\]
  using the convention $\dim \emptyset = -1$.
\end{thm}

\begin{proof}
  We proceed as in the proof of \cref{thm:lowerBound}, but replace the graded $\C$-algebra $\bigoplus_{d \geq 0} H^0(X,\Ell^{\otimes d})$ by $R := \bigoplus_{d \geq 0} H^0(X,\O_X(dD))$. As before, we only need to consider the case that $X$ is irreducible and that $\injDim{X,\O_X(kD)} < \infty$ (in particular, $D$ is ample).
  
  Consider a non-zero subspace $V \subseteq H^0(X,\O_X(kD)) = R_k$ inducing an injective morphism $X \injTo \P(V^*)$ %
  with $\dim \P(V^*) \leq 2\dim X$. With the same notations as in the previous proof, this injection gives rise to a morphism $\psi \colon \Proj R \otimes_\C R \to \P(V^* \times V^*)$ and a linear subspace $L \subseteq \P(V^* \times V^*)$ of dimension $\dim \P(V^*)$ such that
    \[(\psi^{-1} L)^{\text{red}} = \bigcup_{\xi \in \Z_k} (1\times \xi) \cdot Y,\]
  where $Y := V(f \otimes 1 - 1 \otimes f \mid f \in R) \subseteq \Proj R \otimes_\C R$.
  
  In this setting, it remains no longer true that $(1\times \xi) \cdot Y$ and $(1\times \xi') \cdot Y$ are disjoint for $\xi \neq \xi' \in \Z_k$. Instead, we show that
  \begin{equation} \label{eq:intersectTranslates}
    \dim \big((1\times \xi) \cdot Y \cap (1\times \xi') \cdot Y \big) = d_{\ord(\xi {\xi'}^{-1})},
  \end{equation}
  where for each $r \in \N$, we denote by $d_r$ the dimension of $B_r := \bigcap_{r \nmid m} \Bs |mD| \subseteq X$.
  Denoting $\delta := \min\{1+d_q \mid q \text{ prime power dividing } k\}$, we then show that \eqref{eq:intersectTranslates} implies that
  \begin{equation} \label{eq:connectedInDimension}
    \bigcup_{\xi \in \Z_k} (1\times \xi) \cdot Y \text{ is not connected in dimension } \delta.
  \end{equation}
  But on the other hand, it follows from classical connectedness theorems, in particular \cite[Lemma~3.2.2]{FOV99}, that the preimage of the linear subspace $L$ under the finite morphism $\psi \colon \Proj R \otimes_\C R \to \P(V^* \times V^*)$ is connected in dimension
    \[\dim \Proj R \otimes_\C R - \dim L - 2 = 2 \dim X - \dim \P(V^*) - 1,\]
  so we deduce that $\dim \P(V^*) \geq 2 \dim X - \delta$.
  
  It remains to prove \eqref{eq:intersectTranslates} and \eqref{eq:connectedInDimension}. For \eqref{eq:intersectTranslates}, we may restrict to the case $\xi' = 1$ and we denote $r := \ord(\xi)$. Since $(1 \times \xi) \cdot Y = V(f \otimes 1 - 1 \otimes \xi^m f \mid f \in R_m, m \geq 0)$, we have
    \[Y \cap (1 \times \xi) \cdot Y = Y \cap V(R_m \otimes_\C 1 + 1 \otimes_\C R_m \mid m \geq 0 \text{ with } \ord(\xi) \nmid m).\]
  Consider the commutative diagram
  \begingroup \footnotesize
  \[\begin{tikzcd}[column sep=small]
  &
  \Spec R \otimes_\C R
  \arrow[dashed]{dl}
  \arrow[dashed]{dr}
  \arrow{d}{\hat{\psi}}
  &
  \\
  \Proj R \times \Proj R
  \arrow{d}[swap]{\varphi_V \times \varphi_V}
  &
  V^* \times V^*
  \arrow[dashed]{dl}
  \arrow[dashed]{dr}
  &
  \Proj R \otimes_\C R
  \arrow{d}{\psi}
  \\
  \P(V^*) \times \P(V^*)
  &
  &
  \P(V^* \times V^*).
  \end{tikzcd}\]
  \endgroup
  Note that in $\Proj R \cong X$, we have $V(R_m) = \Bs |mD|$. Hence, the affine cone over $Y \cap (1 \times \xi) \cdot Y$ in $\Spec R \otimes_\C R$ is the diagonal $\Delta_{\hat{B_r} \times \hat{B_r}} \subseteq \hat{B_r} \times \hat{B_r}$, where $\hat{B_r} \subseteq \Spec R$ is the affine cone over $B_r \subseteq \Proj R$. In particular, $\dim Y \cap (1 \times \xi) \cdot Y = d_r$, proving \eqref{eq:intersectTranslates}.
  
  In order to show \eqref{eq:connectedInDimension}, let $q = p^\ell$ be a prime power dividing $k$. 
  Let $\zeta$ be a primitive $k$-th root of unity. For $i \in \{1,\ldots,p\}$ consider the set $Z_i := \bigcup_{j = 1}^{k/p} (1 \times \zeta^{jp+i}) \cdot Y \subseteq \Proj R \otimes_\C R$. Then we have the following equalities of sets:
    \[\bigcup_{\xi \in \Z_k} (1\times \xi) \cdot Y = \bigcup_{i=1}^p Z_i
    \quad \text{and} \quad 
     Z_{i_1} \cap Z_{i_2} = \bigcup_{j_1 = 0}^{k/p} \bigcup_{j_2 = 0}^{k/p} \big((1 \times \zeta^{j_1p+i_1}) \cdot Y \cap (1 \times \zeta^{j_2p+i_2}) \cdot Y\big).\]
  Note that for $i_1 \neq i_2 \in \{1,\ldots,p\}$, the order of $\zeta^{(j_1-j_2)p+(i_1-i_2)}$ is divisible by $q$. Since $d_{rq} \leq d_q$ for all $r \geq 1$, we conclude from \eqref{eq:intersectTranslates} that $\dim Z_{i_1} \cap Z_{i_2} \leq d_q$ for all $i_1 \neq i_2$. Hence, $\bigcup_{\xi \in \Z_k} (1\times \xi) \cdot Y$ is not connected in dimension $d_q + 1$. This establishes \eqref{eq:connectedInDimension} and concludes the proof.
\end{proof}

Our situation in the proof of \cref{thm:lowerBoundReflexive} is very close to the setting in \cite{Rei18}, where small separating sets of invariants for finite group actions on affine varieties are investigated. In fact, we remark that \cref{thm:lowerBoundReflexive} could also be obtained as a consequence of \cite[Theorem~4.5]{Rei18}, if we additionally assumed that the section ring $R = \bigoplus_{d \geq 0} H^0(X,\O_X(dD))$ is integrally closed.

As an example, we apply \cref{thm:lowerBoundReflexive} to weighted projective spaces. With the link between injection dimensions of weighted projective spaces and separating invariants of finite cyclic groups (see \cref{ex:wpsAsSepInv}), the following bound on injection dimensions could also be proved on the basis of \cite[Theorem~3.4]{DJ14}. For the purpose of illustration, we choose to base our proof on \cref{thm:lowerBoundReflexive}.

\begin{cor} \label{cor:weightedProjectiveBounds}
  Consider a weighted projective space $\P(q_0,\ldots,q_n)$ that is well-formed, i.e., $\gcd(q_0, \ldots, \widehat{q_i}, \ldots, q_n) = 1$ for all $i$. Let $\ell \geq 2$ be minimal such that
  \[\lcm(q_{i_1},\ldots,q_{i_\ell}) = \lcm(q_0,\ldots,q_n) \qquad \text{holds for all  $i_1, \ldots, i_\ell$ distinct}.\]
  Let $\Ell$ be the ample line bundle generating the Picard group. Then
  \[\injDim{\P(q_0,\ldots,q_n),\Ell^{\otimes k}} \geq
  \begin{cases} 
  n + \ell - 2 &\text{if } k = 1, \\ 2n &\text{if } k \geq 2. \end{cases}\]
\end{cor}

\begin{proof}
  The weighted projective space $X := \P(q_0,\ldots,q_n)$ is $\Proj R$, where $R = \C[x_0,\ldots,x_n]$ is graded via $\deg(x_i) = q_i$. The twisting sheaf $\O_{\Proj R}(1)$ is the reflexive sheaf $\O_X(D)$ on $X$ corresponding to a Weil divisor $D$ generating the class group. For all $m \in \Z$, we have $H^0(X,\O_X(mD)) = R_m$. The ample line bundle generating $\Pic(X)$ is $\Ell = \O_{\Proj R}(a) = \O_X(aD)$, where $a := \lcm\{q_0,\ldots,q_n\}$.
  
  \cref{thm:lowerBound} shows that $\injDim{X,\Ell^{\otimes k}} \geq 2n$ for all $k \geq 2$. We are therefore only concerned with establishing a lower bound for $\injDim{X,\Ell}$ based on \cref{thm:lowerBoundReflexive}. 
  We may assume that not all weights $q_i$ are equal to $1$ (otherwise, $X = \P^n$ and the claim $\injDim{\P^n,\O(1)} \geq n$ is trivial). Then $a \geq 2$ and, by minimality of $\ell$, there exists a prime power $p^r$ dividing $a$ and weights $q_{i_1}, \ldots, q_{i_{\ell-1}}$ not divisible by $p^r$. Note that
  \[\bigcap_{p^r \nmid m} \Bs |mD| = V(R_m \mid p^r\!\nmid\!m) \subseteq V(x_{i_1},\ldots,x_{i_{\ell-1}}).\]
  We conclude with \cref{thm:lowerBoundReflexive} that
  \[\injDim{X,\O_X(aD)} \geq 2n-1-\dim \big(\bigcap_{p^r \nmid m} \Bs |mD|\big) \geq n+\ell-2. \qedhere\]
\end{proof}

\begin{ex} \label{ex:minNotVeryAmple}
  The weighted projective space $\P(1,6,10,15)$ injects to $\P^4$ (see \cref{ex:injNotVeryAmple} and \cref{thm:weightedProjectiveConstr}), but not by linear projections from any embedding. This follows from \cref{thm:lowerBound} because the ample line bundle $\O(30)$ generating the Picard group is not very ample (see \cref{ex:injNotVeryAmple}), whence all very ample line bundles admit a root of some order in $\Pic(X) \cong \Z$.
\end{ex}

The case $\ell = 2$ in \cref{cor:weightedProjectiveBounds} is the case of (non-weighted) projective spaces, for which explicit constructions in the context of separating invariants \cite[Proposition~5.2.2]{Duf08} show that the above bound is sharp. In the next case, $\ell = 3$, we show in \cref{thm:weightedProjectiveConstr} that the bound in \cref{cor:weightedProjectiveBounds} remains sharp when one of the weights is $1$. However, the next example shows that the latter assumption cannot be dropped.

\begin{ex} \label{ex:wpsBoundNotSharp}
  \cref{cor:weightedProjectiveBounds} shows that $\injDim{\P(2,2,3,3),\O(6)} \geq 4$. We show that actually $\injDim{\P(2,2,3,3),\O(6)} = 5$ holds. Indeed, 
  note that the global sections of $\O(6)$ are just $\Sym^3 (\C^2)^* \oplus \Sym^2 (\C^2)^*$ and the secant locus of $Y := \im(\varphi_{\O(6)}) \subseteq \P(\Sym^3 \C^2 \oplus \Sym^2 \C^2)$
  is
  \[\sigma_2^\circ(Y) = \{[v\oplus w] \mid (v = 0 \text{ or } [v] \in \sigma_2^\circ(C_3)) \text{ and } (w = 0 \text{ or } [w] \in \sigma_2^\circ(C_2))\},\]
  where $C_d \subseteq \P \Sym^d \C^2$ for $d\in\{2,3\}$ denotes the $d$-th rational normal curve. The secant locus of the plane conic $C_2$ fills its ambient space, while the secant locus of the twisted cubic curve $C_3$ consists of all points that cannot be written as $[v_1 v_2^2] \in \P(\Sym^3 \C^2)$ with $\{v_1, v_2\}$ a basis of $\C^2$. Hence, a linear subspace $L \subseteq \P(\Sym^3 \C^2 \oplus \Sym^2 \C^2)$ does not meet $\sigma_2^\circ(Y)$ if and only if it does not meet the center of the projection $\pi_1 \colon \P(\Sym^3 \C^2 \oplus \Sym^2 \C^2) \dashrightarrow \P(\Sym^3 \C^2)$ and $\pi_1(L) \cap \sigma_2^\circ(C_3) = \emptyset$. Every line in $\P(\Sym^3 \C^2)$ intersects $\sigma_2^\circ(C_3)$ by \cref{prop:secantAvoidance} and \cref{rem:dimPlusOne}, so $\pi_1(L)$ and hence $L$ needs to be a point. In fact, $L$ can be chosen to be the point $[v_1 v_2^2 \oplus 0] \notin \sigma_2^\circ(Y)$ (for some basis $\{v_1, v_2\}$ of $\C^2$). By \cref{prop:secantAvoidance}, this shows $\injDim{\P(2,2,3,3),\O(6)} = 5$.
\end{ex}

\cref{ex:wpsBoundNotSharp} generalizes easily to show that an $n$-dimensional weighed projective space of the form $\P(2,\ldots,2,d,d)$ with $n,d \geq 3$ and $d$ odd cannot be injected to $\P^{2n-2}$, while the bound from \cref{cor:weightedProjectiveBounds} only established that injections to $\P^{2n-3}$ are impossible. In fact, we always expect the following:

\begin{conjecture} \label{conj:twoWeightsWPS}
  Let $d,e \geq 2$ be relatively prime and let $r,s \geq 2$. Then 
    \[\injDim{\P(\underbrace{d,\ldots,d}_{r},\underbrace{e,\ldots,e}_{s})} = 2(r+s-1)-1.\]
\end{conjecture}

From $\injDim{\P^{r-1},\O(e)} = 2(r-1)$ and $\injDim{\P^{s-1},\O(d)} = 2(s-1)$, see \cite[Proposition~5.2.2]{Duf08}, the existence of an injection $\P(d,\ldots,d,e,\ldots,e) \injTo \P^{2(r+s-1)-1}$ is clear. The result in question in \cref{conj:twoWeightsWPS} is therefore the lower bound improving the one obtained from \cref{cor:weightedProjectiveBounds}.

We finish this section by applying our lower bounds on injection dimensions also to a non-toric example.

\begin{ex} \label{ex:noSmallInjectionSmooth}
  Let $n \geq 3$ and let $q_0,\ldots,q_{n+1} \geq 2$ be pairwise relatively prime. Let $X \subseteq \P(q_0,\ldots,q_{n+1})$ be a general hypersurface of weighted degree $d := q_0 q_1 \cdots q_{n+1}$. Then $X$ is an $n$-dimensional connected smooth projective variety with $\injDim{X} \geq 2n$. Indeed, \cite[Theorem~1]{RS06} implies that the restriction morphism of class groups $\Cl(\P(q_0,\ldots,q_n)) \to \Cl(X)$ is an isomorphism. Since $\P(q_0,\ldots,q_n)$ has only finitely many singular points, $X$ is smooth, so we conclude that $\Pic(X) = \Cl(X)$ is generated by the restriction of the reflexive sheaf $\O(1)$ on $\P(q_0,\ldots,q_{n+1})$ to $X$. On the other hand, by generality of $X$, the homomorphism $H^0(\P(q_0,\ldots,q_n), \O(1)) \to H^0(X,\restr{\O(1)}{X})$ is surjective, so the line bundle $\restr{\O(1)}{X}$ has no global sections. In particular, it cannot give rise to injective morphisms. Every other line bundle is a power of $\restr{\O(1)}{X}$, so \cref{thm:lowerBound} implies that $\injDim{X} \geq 2 \dim X$.
\end{ex}

\section{Explicit constructions of injections}
In this section, we give three specific approaches for explicitly constructing injective morphisms $X \to \P^{2 \dim X}$ with a focus on products of projective spaces and weighted projective spaces.

\subsection{Constructions from tangential varieties} \label{ssec:TangConstr}

The following approach is helpful for producing injections of $\P^m \times \ldots \times \P^m$ for $m \geq 1$:
Consider $\dVec = (d_1,\ldots,d_r) \in \Z_{>0}^r$ and let $D := \sum_{i=1}^r d_i$. The image of the morphism
\begin{align*}
  \psi_{m,\dVec} \colon \P(\C^{m+1}) \times \ldots \times \P(\C^{m+1}) &\to \P(\Sym^{D} \C^{m+1}), \\
  [v_1] \times \ldots \times [v_r] &\mapsto  [v_1^{d_1} \cdots v_r^{d_r}]
\end{align*}
is the \defstyle{Chow--Veronese variety} of type $(m,\dVec)$. In the case $m = 1$, this is also called the \defstyle{coincident root locus} of type $\dVec$, and its understanding is of major interest from the viewpoint of practical applications, see for example \cite{LS16}. We observe that injections of these varieties for suitable $\dVec$ give rise to injections of products of projective spaces:
  
\begin{lemma} \label{lem:ChowVeroneseParam}
  Let $r, m \in \Z_{>0}$ and let $\dVec \in \Z_{>0}^r$ be such that
  \begin{equation} \label{eq:distinctDAssumption}
    \sum_{i \in I} d_i \neq \sum_{j \in J} d_j \qquad \text{for all } I, J \subseteq [r] \text{ with }I \cap J = \emptyset.
  \end{equation}
  Then the morphism $\psi_{m,\dVec}$ is injective. In particular, if $Y \subseteq \P(\Sym^D \C^{m+1})$ denotes the Chow--Veronese variety of type $(m,\dVec)$, then
    \[\injDim{\P^m \times \ldots \times \P^m,\O(k \dVec)} \leq \injDim{Y,\O_Y(k)} \qquad \text{for all } k > 0.\]
\end{lemma}

\begin{proof}
  Let $p \in Y$, then $p = [z]$ for $z \in \Sym^{D} \C^{m+1} \setminus \{0\}$ of the form $z = v_1^{d_1} \cdots v_r^{d_r}$ for some $v_i \in \C^{m+1} \setminus \{0\}$. To show injectivity of $\psi_{m,\dVec}$, we need to see that each $v_i$ is uniquely determined up to scaling. But $\Sym^{\bullet} \C^{m+1} \cong \C[x_0,\ldots,x_m]$ is a unique factorization domain, so $z$ uniquely determines the set $\{[v_1],\ldots,[v_r]\} \subseteq \P \C^{m+1}$ as the linear factors of $z$ up to scaling. Moreover, for each $i$, the factor $v_i \in \C^{m+1}$ appears in $z$ with multiplicity $m_i := \sum_{j \in J_i} d_j$, where $J_i := \{j \mid v_j = v_i\}$. By assumption \eqref{eq:distinctDAssumption}, the integer $m_i$ uniquely determines the set $J_i$. Therefore, each $v_i$ is uniquely determined up to scaling from $z$. Hence, $\psi_{m,\dVec}$ is injective. The second claim follows from  $\psi_{m,\dVec}^*(\O_Y(1)) = \O(\dVec)$.
\end{proof}

In the case $r = 2$, $\dVec = (1,d-1)$ with $d \geq 3$, the Chow--Veronese variety of type $(m,\dVec)$ is 
the tangential variety of the $d$-th Veronese variety $\nu_d(\P^m)\subseteq  \P(\Sym^d \C^{m+1})$. As a simple consequence of \cref{lem:ChowVeroneseParam}, we construct two explicit injections from the cases in which this tangential variety has a secant variety of small dimension (as classified in \cite{CCG02} and \cite{AV18}).

\begin{prop} \label{prop:tangentialInjections}
  The two morphisms
  \begin{align*}
    \P^1 \times \P^1 &\to \P^3 \\
    \scriptstyle [x_0:x_1]\times[y_0:y_1] &
    \scriptstyle \; \mapsto \;
    [x_0 y_0^2 \: : \: x_1 y_0^2+2x_0 y_0 y_1 \: : \: 2 x_1 y_0 y_1 + x_0 y_1^2 \: : \: x_1 y_1^2], \\[0.5em]
    \P^2 \times \P^2 &\to \P^8, \\
    \scriptstyle
    [x_0:x_1:x_2]\times[y_0:y_1:y_2] &
    \scriptstyle \; \mapsto \;
    [x_0y_0^2 \: : \:
     x_1 y_1^2 \: : \:
     x_2y_2^2 \: : \:     
     x_0y_1^2+2x_1y_0y_1 \: : \:
     x_1y_2^2+2x_2y_1y_2 \: : \:
     x_2y_0^2+2x_0y_0y_2 \: : \:
     2x_0y_1y_2+2x_1y_0y_2+2x_2y_0y_1 \: : \:
    \\
    &\phantom{\scriptstyle \; \mapsto \; [}
    \scriptstyle
    x_1y_0^2-x_2y_1^2+2x_0y_0y_1-2x_1y_1y_2 \: : \:
    x_1y_0^2-x_0y_2^2+2x_0y_0y_1-2x_2y_0y_2]
  \end{align*}
  are injective. In particular, $\injDim{\P^1 \times \P^1, \O(1,2)} = 3$ and $\injDim{\P^2 \times \P^2, \O(1,2)} \leq 8$.
\end{prop}

\begin{proof}
  By \cref{lem:ChowVeroneseParam}, the morphism $\psi_{1,(1,2)} \colon \P \Sym^1 \C^2 \times \P \Sym^1 \C^2 \to \P \Sym^3 \C^2$ is injective and its image is the tangential surface of the twisted cubic curve in $\P \Sym^3 \C^2$. Picking a basis $\{T_0,T_1\}$ of $\C^2$, we consider the  $\{T_0^3,T_0^2T_1,T_0T_1^2,T_1^3\}$ as a basis of $\Sym^3 \C^2$. In these bases, the morphism $\psi_{1,(1,2)} \colon \P^1 \times \P^1 \injTo \P^3$ maps $[x_0:x_1] \times [y_0:y_1]$ to the point whose coordinates are the coefficients in $T_0$ and $T_1$ of the expression $(x_0T_0+x_1T_1)(y_0T_0 + y_1T_1)^2$. Explicitly, this is the injection written out above. Note that it is given by global sections of $\O(1,2)$.
  
  Similarly, the morphism $\psi_{2,(1,2)}$ injectively maps to the tangential variety of the Veronese surface and it is known that the secant variety of this tangential fourfold is of dimension~$8$ only, see \cite[Proposition~3.2]{CCG02}. By \cref{lem:secantAvoidance}, a projection of the image of $\psi_{2,(1,2)}$ from a general linear space of codimension~$9$ gives an injection $\P^2 \times \P^2 \injTo \P^8$ given by global sections of $\O(1,2)$.
  
  Explicitly, with respect to a basis $\{T_0,T_1,T_2\} \subseteq \C^3$ and the corresponding basis $\{T_i T_j T_k \mid i,j,k\} \subseteq \Sym^3 \C^3$,
  the injection $\psi_{2,(1,2)}$ is described as $\P^2 \times \P^2 \injTo \P^9$, mapping $[x_0:x_1:x_2] \times [y_0:y_1:y_2]$ to the $10$~coefficients in $T_0,T_1,T_2$ of the expression  $(x_0T_0+x_1T_1+x_2T_2)(y_0T_0+y_1T_1+y_2T_2)^2$. The secant variety of its image in $\P \Sym^3 \C^3$ does not fill the entire $9$-dimensional projective space. Explicitly, one checks that $p := [T_0^2 T_1+T_1^2 T_2+T_2^2 T_1] \in \P(\Sym^3 \C^3)$ does not lie on a secant line. In particular, the composition of $\psi_{2,(1,2)}$ with the projection from $p$ gives an injective morphism $\P^2 \times \P^2 \injTo \P^8$. This is the morphism written out above.
\end{proof}

Another application of \cref{lem:ChowVeroneseParam} is the following explicit construction:

\begin{prop} \label{prop:injectP1P1}
  Let $d \geq 3$. Then the following morphism is injective:
    \begin{align*}
      \P^1 \times \P^1 &\injTo \P^4, \\
      \scriptstyle [x_0:x_1] \times [y_0:y_1] & \ \mapsto\  \scriptstyle [x_0 y_0^d \: : \: dx_0 y_0^d y_1+ x_1 y_0^d \: : \: \binom{d}{2} x_0 y_0^{d-1} y_1^2 + d x_1 y_0^d y_1 \: : \: x_0 y_1^d + d x_1 y_0 y_1^{d-1} \: : \: x_1 y_1^d].
    \end{align*}
  In particular, $\injDim{\P^1 \times \P^1, \O(1,d)} \leq 4$.
\end{prop}

\begin{proof}
  We use \cref{lem:ChowVeroneseParam} and construct an injection of the tangential variety of the rational normal curve of degree~$d+1$: The morphism $\psi_{1,(1,d)} \colon \P^1 \times \P^1 \to \P \Sym^{d+1} \C^2$ is injective. Fixing a basis $\{T_0,T_1\} \subseteq \C^2$ and the corresponding basis $\{T_{i_0}\cdots T_{i_d}\} \subseteq \Sym^{d+1} \C^2$, it is explicitly given by mapping $[x_0:x_1] \times [y_0:y_1]$ to the $d+1$ coefficients in $T_0,T_1$ of the expression $(x_0T_0+x_1T_1)^d (y_0T_0+y_1 T_1)$. The morphism written out above only maps to the point in $\P^4$ whose coordinates are the coefficients of the monomials $T_0^{d+1}$, $T_0^d T_1$, $T_0^{d-1} T_1^2$, $T_0 T_1^d$ and $T_1^{d+1}$. This is the composition of $\psi_{1,(1,d)}$ with the projection $\P(\Sym^{d+1} \C^2) \to \P(\Sym^{d+1} \C^2/W)$, where $W := \langle T_0^{d-2} T_1^3, \ldots, T_0^2 T_1^{d-1} \rangle$.
  
  By \cref{lem:secantAvoidance}, we need to see that $\P(W)$ does not meet the secant locus of the tangential variety of the rational normal curve of degree~$d+1$. We check this explicitly: Assume that $f^d g - {f'}^d g' \in W \setminus \{0\}$ for some $f,f',g,g' \in \C^2 \setminus \{0\}$. Since every element of $W$ is divisible by $T_0^2 T_1^2$, we note that $f$ is proportional to $T_0$ (resp.\ $T_1$) if and only if $f'$ is. But this cannot be the case, since no nonzero element of $W$ is divisible by $T_0^d$ (resp.\ $T_1^d$). Similarly, $g$ is proportional to $T_0$ (resp.\ $T_1$) if and only if $g'$ is. Since $W_0 := \langle T_0^{d-2} T_1^2, \ldots, T_0^2 T_1^{d-2}, T_0 T_1^{d-1} \rangle \subseteq \Sym^d \C^2$ defines a linear space $\P(W_0)$ not intersecting the secant locus of the rational normal curve $\nu_d(\P^1) \subseteq \P \Sym^d \C^2$, no nonzero element of $W \subseteq T_0 W_0 \cap T_1 W_0$ can be written as $T_0(f^d-{f'}^d)$ or $T_1(f^d-{f'}^d)$.
  
  We may therefore now assume that $f,f',g,g'$ are not proportional to $T_0, T_1$, so we can write (after rescaling):
    \[f = T_0-\alpha T_1, \qquad g = T_0-\beta T_1, \qquad f' = T_0 - \alpha' T_1, \qquad g' = T_0 - \beta' T_1\]
  for non-zero $\alpha, \alpha', \beta, \beta' \in \C$.
  Since $f^d g - {f'}^d g' \in W$, it has zero coefficients for $T_0^{d+1}$, $T_0^d T_1$, $T_0^{d-1} T_1^2$, $T_0 T_1^d$ and $T_1^{d+1}$, i.e.\ 
    \[d\alpha+\beta = d\alpha'+\beta',
    \ {\textstyle \binom{d}{2} }\alpha^2 + d \alpha \beta = {\textstyle \binom{d}{2} }{\alpha'}^2 + d \alpha \beta',
    \ \alpha^d + d\alpha^{d-1} \beta = {\alpha'}^d + d{\alpha'}^{d-1} \beta',
    \ \alpha^d \beta = {\alpha'}^d \beta'\]
  We wish to conclude $\alpha = \alpha'$ and $\beta = \beta'$. As
    \[d\alpha^2+\beta^2 = (d\alpha+\beta)^2 - 2\left({\textstyle \binom{d}{2} }\alpha^2 + d \alpha \beta\right)
    \quad \text{and} \quad
    \frac{d}{\alpha} + \frac{1}{\beta} = \frac{\alpha^d + d\alpha^{d-1} \beta}{\alpha^d \beta},\]
  this follows from the following lemma.
\end{proof}

\begin{lemma}
  For $d \geq 2$, the morphism $(\C^*)^2 \to \C^3$, $(x,y) \mapsto (x^{-1}+dy^{-1},x+dy,x^2+dy^2)$ is injective.
\end{lemma}

\begin{proof}
  The fiber over a point $(a,b,c) \in \C^3$ is given by the vanishing of 
    \[f := x^{-1}+dy^{-1} - a,\qquad g := x+dy-b \qquad \text{and} \qquad h := x^2+dy^2 - c.\]
  The ideal $(f,g,h) \subseteq \C[x^{\pm 1},y^{\pm 1}]$ contains
    \[(d+1)xyf + a(x+y)g - ah = (d(d+1)-ab)x+(d+1-ab)y+ac.\]
  If the fiber consisted of more than one point, then this linear polynomial needs to be proportional to the linear polynomial $g$. Then in particular $d+1-ab = d(d(d+1)-ab)$, i.e., $ab = (d+1)^2$. In this case, 
    \[axyf + (a^2y-ad) g 
     = a^2dy^2+(a-ad^2-a^2b)y-abd
      = d(ay-d-1)^2,
    \]
  hence $ay-d-1 = 0$ for all points in $(\C^*)^2$ on which $f,g,h$ vanish. Then the unique point mapping to $(a,b,c)$ is $(\frac{d+1}{a},\frac{d+1}{a})$.
\end{proof}

\subsection{Inductive constructions} \label{ssec:IndConstr}

In the following, we provide explicit injections of some weighted projective spaces matching the lower bounds on injection dimensions in \cref{cor:weightedProjectiveBounds}. This generalizes the case of (non-weighted) projective spaces worked out in \cite[Proposition~5.2.2]{Duf08}.

\begin{thm} \label{thm:weightedProjectiveConstr}
  Consider a weighted projective space $\P(q_0,q_1,\ldots,q_n)$ with $q_0 = 1$ and $\lcm(q_i,q_j,q_k) = d$ for all distinct $i,j,k$, where $d := \lcm(q_1,\ldots,q_n)$. Let $a_i := d/q_i$ and $b_{ij} := \lcm(q_i,q_j)/q_i$ for all $i,j \in \{0,1,\ldots,n\}$. Then the following are injective morphisms:
  \begin{align*}
    \varphi_1 \colon \P(q_0,q_1,\ldots,q_n) &\to \P^{n+1}, \\
    [x_0:\ldots:x_n]&\mapsto [x_0^d : x_0^{d-q_1} x_1 : x_1^{a_1} + x_0^{d-q_2} x_2 : x_2^{a_2} + x_0^{d-q_3} x_3 : \\
    &\hspace{9em} \ldots : x_{n-1}^{a_{n-1}} + x_0^{d-q_n} x_n : x_n^{a_n}], \\
  \text{for }k \geq 2: \quad \varphi_k \colon \P(q_0,q_1,\ldots,q_n) &\to \P^{2n}, \\
  [x_0:\ldots:x_n] &\mapsto \big[\sum_{\substack{i+j=\ell \\ i \leq j}} x_i^{ka_i-b_{ij}} x_j^{b_{ji}} \mid \ell=0,1,\ldots,2n\big].
  \end{align*}
  In particular, the injection dimensions for these weighted projective spaces are as follows:
    \[\injDim{\P(q_0,q_1,\ldots,q_n), \O(kd)} = 
    \begin{cases}
      \infty  &\text{if } k \leq 0, \\
      n       &\text{if } k = 1 \text{ and } q_1=\ldots=q_n, \\
      n+1     &\text{if } k = 1 \text{ and $q_1,\ldots,q_n$ not all equal}, \\
      2n      &\text{if } k \geq 2.
    \end{cases}\]
\end{thm}

\begin{proof}
  Recall that for any $r_0,\ldots,r_m \in \Z_{>0}$ and $c \in \Z_{>0}$ with $(r_0,c)=1$, there is an isomorphism
  \begin{align}
    \P(r_0,c r_1,\ldots,c r_n) &\cong \P(r_0,r_1,\ldots,r_m) , \quad [x_0:x_1:\ldots:x_m] \mapsto [x_0^c:x_1:\ldots:x_m].
    \label{eq:wellFormingWPS} 
  \end{align}
  In particular, we may assume that $\gcd(q_1, \ldots, q_n) = 1$, noting that the descriptions of $\varphi_1$ and $\varphi_k$ do not change under composition with the isomorphism \eqref{eq:wellFormingWPS}.
  
  The restriction of $\varphi_1$ to $V(x_0) \cong \P(q_1,\ldots,q_n)$ is given by $\varphi_1([0:x_1:\ldots:x_n]) = [0:0: x_1^{a_1}: x_2^{a_2} : \ldots : x_n^{a_n}]$. By assumption, $\lcm(q_i,q_j) = \lcm(q_0,q_i,q_j) = d$ for all positive $i \neq j$. In particular, if $p^r$ is a prime power in $d$ not dividing $q_i$ for some $i > 0$, then $p^r$ must divide all $q_j$ for $j>0$, $j \neq i$. Composing the corresponding isomorphisms \eqref{eq:wellFormingWPS}, we see that $\P(q_1,\ldots,q_n) \to \P^n$, $[x_1:\ldots:x_n] \mapsto [x_1^{a_1}: x_2^{a_2} : \ldots : x_n^{a_n}]$ is an isomorphism. In particular, $\restr{\varphi_1}{V(x_0)}$ is injective.
  
  The restriction of $\varphi_1$ to the affine open $D(x_0)$ can be checked to be injective by setting $x_0 = 1$.
  We have
  \begin{align*}
    \restr{\varphi_1}{x_0=1} \colon \A^n &\to \A^{n+1} \\ 
    (x_1,\ldots,x_n) &\mapsto (x_1, x_2+x_1^{a_1}, x_3+x_2^{a_2}, \ldots, x_n+x_{n-1}^{a_{n-1}},x_n^{a_n}),
  \end{align*}
  which is a closed embedding, since it is of triangular shape.
  Note that for points in $\varphi(V(x_0))$ the first coordinate is zero, while for points in $\varphi(D(x_0))$ it does not. Hence, the images of $\restr{\varphi}{V(x_0)}$ and $\restr{\varphi}{D(x_0)}$ do not intersect. We conclude that $\varphi$ is injective.
  
  For $\varphi_k$ with $k \geq 2$, we also consider the affine open $D(x_0)$ by setting $x_0=1$, giving
  \begin{align*}
    \restr{\varphi_k}{x_0=1} \colon \A^n &\to \A^{2n} \\ 
    (x_1,\ldots,x_n) &\mapsto (x_1, x_2+x_1^{ka_1}, x_3+x_1^{ka_1-b_{12}}x_2^{b_{21}}, x_4 + x_1^{ka_1-b_{13}}x_3^{b_{31}} + x_2^{ka_2}, \ldots),
  \end{align*}
  which is of triangular shape and therefore a closed embedding. The first coordinate for points in $\varphi_k(V(x_0))$ vanishes, while this is not the case for points in $\varphi_k(D(x_0))$. Hence, with the above, is enough to show that the restriction of $\varphi_k$ to $V(x_0)$ is injective. As above, using \cref{eq:wellFormingWPS}, we have the isomorphism $\restr{\varphi_1}{V(x_0)} \colon V(x_0) \xrightarrow{\cong} \P^n$. Composing with this isomorphism, it only remains to show that
    \[\psi \colon \P^{n-1} \to \P^{2(n-1)}, \qquad [x_1:\ldots:x_n] \mapsto \big[\sum_{\substack{i+j=\ell \\ i \leq j}} x_i^{k-1} x_j \mid \ell=2,3\ldots,2n\big] \]
  is injective. This was originally proved in \cite[Proposition~5.2.2]{Duf08}. It follows by induction on $n$ as follows: The case $n=1$ is trivial. Let $n \geq 2$. The restriction of $\psi$ to $V(x_1) \cong \P^{n-2}$ is injective by the induction hypothesis. The restriction of $\psi$ to the affine open $D(x_1)$ can be checked to be injective by setting $x_1 = 1$. We have
    \begin{align*}
      \restr{\varphi}{x_1=1} \colon \A^{n-1} &\to \A^{2(n-1)} \\ 
      (x_2,\ldots,x_n) &\mapsto (x_2, x_3+x_2^k, x_4+x_2^{k-1}x_3, x_5+\ldots, \ldots),
    \end{align*}
  which is a closed embedding. Since $\psi(V(x_1)) \cap \psi(D(x_1)) = \emptyset$, we conclude that $\psi$ is injective.
\end{proof}

We employ a similar technique to construct an injective morphism $\P^1 \times \P^n \injTo \P^{2(n+1)}$ given by global sections of $\O(d,1)$:

\begin{thm} \label{thm:injectP1Pn}
  Let $n, d \geq 1$. The following is an injective morphism:
  \begin{align*}
  \P^1 \times \P^n {}\injTo{} &\P^{2(n+1)}, \\
  [x_0:x_1] \times [y_0:\ldots:y_n] \mapsto{} 
  & [ x_0^d y_0: x_0^d y_1: \ldots : x_0^d y_{n-1}:  x_0^d y_n : \quad x_1^d y_n : dx_0 x_1^{d-1} y_0 :   \\
  & \phantom{[ }x_1^d y_0+dx_0 x_1^{d-1} y_1 :  x_1^d y_1+dx_0 x_1^{d-1} y_2 : \ldots : x_1^d y_{n-1}+dx_0 x_1^{d-1} y_n].
  \end{align*}
  In particular, $\injDim{\P^1 \times \P^n, \O(d,1)} \leq 2(n+1)$.
\end{thm}

\begin{proof}
  The morphism $\varphi \colon \P^1 \times \P^n \injTo \P^{2n+2}$ written out above maps $[x_0:x_1] \times [y_0:\ldots:y_n]$ to the point in $\P^{2n+2}$ whose coordinates are the coefficients of $T^{d+1}$ and $TS^d$ in the expressions $\{f_i := (x_0 T+x_1 S)^d (y_i T + y_{i+1}S) \mid i=-1,0,1,\ldots,n\}$, where $y_{-1} := y_{n+1} := 0$. Restricting $\varphi$ to $[0:1] \times \P^n$, we obtain the closed embedding
    \[[0:1] \times \P^n \hookrightarrow \P^{2n+2}, \quad [0:1] \times [y_0:\ldots:y_n] \mapsto [0:\ldots:0:y_n:0:y_0:y_1:\ldots:y_{n-1}].\]
  Note that the points in the image have the first $n+1$ coordinates zero, hence the image of $[0:1] \times \P^n$ does not intersect $\varphi(D(x_0) \times \P^n)$. In particular, it is sufficient to show injectivity of the restriction of $\varphi$ to $D(x_0) \times \P^n$, for which we may simply set $x_0=1$. The first of the expressions $f_i = (x_0T+x_1S)^d (y_i T + y_{i+1}S)$ with a non-zero coefficient of $T^{d+1}$ uniquely determines the minimal $k$ such that $y_k \neq 0$. Similarly, the last expression in which $TS^d$ appears with non-zero coefficient determines the maximal $m$ with $y_m \neq 0$. Hence, for a point in the image of
    \[\restr{\varphi}{x_0=1} \colon \A^1 \times \P^n \to \P^{2n+2}, \ [1:x_1] \times [y_0: \ldots : y_n] \mapsto [y_0:\ldots : y_n: x_1^d y_n : d x_1^{d-1} y_0:\ldots ],\]
  the values $k$ and $m$ can be read off the zero-pattern of its coordinates. Note also that $y_0, \ldots, y_n$ are determined by the first $n+1$ coordinates. Finally, $x_1$ can be reconstructed from the coordinates as
    \[x_1 = d \: \frac{y_k}{y_m} \: \frac{x_1^d y_m}{dx_1^{d-1} y_k} = d \: \frac{x_0^d y_k}{x_0^d y_m} \: \frac{x_1^d y_m + dx_0 x_1^{d-1} y_{m+1}}{x_1^d y_{k-1} + dx_0 x_1^{d-1} y_{k}}.\]
  This shows that any point in the image of $\restr{\varphi}{x_0=1}$ determines a unique preimage, proving injectivity.
\end{proof}

Note that for $n = 1$, \cref{thm:injectP1Pn} gives another injection of $\P^1 \times \P^1 \injTo \P^4$ via global sections of $\O(d,1)$, structurally different from the one constructed before in \cref{prop:injectP1P1}.

\subsection{Graph-theoretic constructions} \label{ssec:GraphConstr}

In this section, we give a combinatorial construction of an injection $\P^1 \times \P^1 \times \P^m \hookrightarrow \P^{2m+4}$ by multilinear forms, showing that $\injDim{\P^1 \times \P^1 \times \P^m,\O(1,1,1)} \leq 2m+4$.

The complete linear system $|\O(1,1,1)|$ embeds $\P^1 \times \P^1 \times \P^m$ by the Segre embedding 
  \[\P(\C^2) \times \P(\C^2) \times \P(\C^{m+1}) \hookrightarrow \P(\C^2 \otimes \C^2 \otimes \C^{m+1}), \quad [u] \times [v] \times [w] \mapsto [u \otimes v \otimes w].\]
We denote its image by $Y \subseteq \P(\C^2 \otimes \C^2 \otimes \C^{m+1})$. For any $k \geq 1$, we denote the standard basis vectors of $\C^{k+1}$ by $e_0, e_1, \ldots, e_k$, and we write $e_0^*,\ldots,e_k^* \in (\C^{k+1})^*$ for the dual basis.

Contrary to the approach in \cref{ssec:TangConstr} and \cref{ssec:IndConstr}, here, we do not construct the injection by writing out explicit polynomials defining the morphism and proving injectivity by exploiting their structure. Instead, we explicitly describe a linear subspace $L \subseteq \P(\C^2 \otimes \C^2 \otimes \C^{m+1})$ of codimension $2m+3$ not intersecting the secant locus $\sigma_2^{\circ}(Y)$. Then, by \cref{prop:secantAvoidance}, the projection of $Y$ from $L$ gives an injection to $\P^{2m+4}$.

Necessarily, a codimension $2m+3$ linear space must intersect the $(2m+3)$-dimensional secant variety $\sigma_2(Y)$, and we need to ensure that this intersection does not meet the secant locus $\sigma_2^{\circ}(Y)$. In fact, we construct a linear subspace whose intersection with the secant variety is of much higher than expected dimension:

\begin{thm} \label{thm:graphConstr}
  Consider the Segre variety $Y \subseteq \P(\C^2 \otimes \C^2 \otimes \C^{m+1})$. There exists a linear subspace $L \subseteq \P(\C^2 \otimes \C^2 \otimes \C^{m+1})$ of codimension $2m+3$ not meeting $\sigma_2^{\circ}(Y)$ such that 
    \[L \cap \sigma_2(Y) = L_1 \sqcup L_2,\]
  where $L_1,L_2 \subseteq \sigma_2(Y) \setminus \sigma_2^{\circ}(Y)$ are disjoint linear spaces spanning $L$.
\end{thm}

We recall that $Y$ consists of the points corresponding to rank~$1$ tensors in $\C^2 \otimes \C^2 \otimes \C^{m+1}$ and $\sigma_2^{\circ}(Y)$ is the set of rank~$\leq 2$ tensors. The closure of the latter is the secant variety $\sigma_2(Y)$ parameterizing tensors of \emph{border rank} at most~$2$. Then \cref{thm:graphConstr} states the existence of two disjoint large-dimensional linear subspaces consisting of border rank~$2$ tensors only, whose common span does not contain any rank~$2$ tensor.

In order to constructively prove \cref{thm:graphConstr}, we first introduce combinatorial objects encoding in a useful way the tensor subspaces which we will consider.

Throughout, we fix $m\in \Z_{>0}$. Let $\Gamma$ be the directed graph with vertex set $\{0,1,\ldots,m\}$ and edges $E := E_1 \sqcup E_2$, where
\begin{align*}
  E_1 &:= \{(0,1),(1,2),(2,3),\ldots,(m-1,m)\} \qquad \text{and} \\
  E_2 &:= \{(i,m-i) \mid 0 \leq i < \lfloor m/2 \rfloor\} \cup \{(m-i,i+1) \mid 0 \leq i < \lfloor (m-1)/2 \rfloor\} \\
&\phantom{:}= \{(0,m),(m,1),(1,m-1),(m-1,2), \ldots\}.
\end{align*}
See \cref{fig:Graph} for an illustration. Note that $E_1$ and $E_2$ each form a directed path in $\Gamma$.
\begin{figure}
\definecolor{qqzzqq}{rgb}{0.,0.6,0.}
\definecolor{ffqqqq}{rgb}{1.,0.,0.}
\definecolor{ffffff}{rgb}{1.,1.,1.}
\begin{tikzpicture}[line cap=round,line join=round,>=triangle 45,x=1.0cm,y=1.0cm,scale=0.45]
\clip(-4.,-1.) rectangle (7.,9.);
\draw [line width=1.pt,color=ffffff] (0.,0.)-- (3.096192290249435,0.);
\draw [line width=1.pt,color=ffffff] (3.096192290249435,0.)-- (5.026636607813593,2.4207006082917997);
\draw [line width=1.pt,color=ffffff] (5.026636607813593,2.4207006082917997)-- (4.337669007678949,5.439264893538139);
\draw [line width=1.pt,color=ffffff] (4.337669007678949,5.439264893538139)-- (1.548096145124718,6.7826523814585205);
\draw [line width=1.pt,color=ffffff] (1.548096145124718,6.7826523814585205)-- (-1.2414767174295134,5.439264893538141);
\draw [line width=1.pt,color=ffffff] (-1.2414767174295134,5.439264893538141)-- (-1.9304443175641588,2.420700608291801);
\draw [line width=1.pt,color=ffffff] (-1.9304443175641588,2.420700608291801)-- (0.,0.);
\draw [line width=1.pt,color=ffqqqq] (0.,0.)-- (-1.9304443175641588,2.420700608291801);
\draw [line width=1.pt,color=ffqqqq] (-1.0894196307414492,1.3660890080209276) -- (-0.7783357141320475,1.359387270497145);
\draw [line width=1.pt,color=ffqqqq] (-1.0894196307414492,1.3660890080209276) -- (-1.1521086034321104,1.0613133377946569);
\draw [line width=1.pt,color=ffqqqq] (-1.9304443175641588,2.420700608291801)-- (-1.2414767174295134,5.439264893538141);
\draw [line width=1.pt,color=ffqqqq] (-1.5416349529518785,4.124185738156523) -- (-1.352916932806974,3.876792073461022);
\draw [line width=1.pt,color=ffqqqq] (-1.5416349529518785,4.124185738156523) -- (-1.819004102186699,3.983173428368921);
\draw [line width=1.pt,color=ffqqqq] (-1.2414767174295134,5.439264893538141)-- (1.548096145124718,6.7826523814585205);
\draw [line width=1.pt,color=ffqqqq] (0.33278026071779765,6.1973870976945244) -- (0.25702386608303446,5.895593981254097);
\draw [line width=1.pt,color=ffqqqq] (0.33278026071779765,6.1973870976945244) -- (0.04959556161217151,6.326323293742565);
\draw [line width=1.pt,color=ffqqqq] (1.548096145124718,6.7826523814585205)-- (4.337669007678949,5.439264893538139);
\draw [line width=1.pt,color=ffqqqq] (3.12235312327203,6.024530177302137) -- (2.8391684241664037,5.8955939812540965);
\draw [line width=1.pt,color=ffqqqq] (3.12235312327203,6.024530177302137) -- (3.0465967286372666,6.326323293742564);
\draw [line width=1.pt,color=ffqqqq] (4.337669007678949,5.439264893538139)-- (5.026636607813593,2.4207006082917997);
\draw [line width=1.pt,color=ffqqqq] (4.726478372291228,3.7357797636734174) -- (4.449109223056409,3.8767920734610204);
\draw [line width=1.pt,color=ffqqqq] (4.726478372291228,3.7357797636734174) -- (4.915196392436133,3.9831734283689184);
\draw [line width=1.pt,color=ffqqqq] (5.026636607813593,2.4207006082917997)-- (3.096192290249435,0.);
\draw [line width=1.pt,color=ffqqqq] (3.937216977072143,1.0546116002708743) -- (3.874528004381481,1.359387270497145);
\draw [line width=1.pt,color=ffqqqq] (3.937216977072143,1.0546116002708743) -- (4.248300893681544,1.0613133377946569);
\draw [line width=1.pt,color=qqzzqq] (0.,0.)-- (3.096192290249435,0.);
\draw [line width=1.pt,color=qqzzqq] (1.7472934240362819,0.) -- (1.548096145124717,-0.23903673469387746);
\draw [line width=1.pt,color=qqzzqq] (1.7472934240362819,0.) -- (1.548096145124717,0.23903673469387746);
\draw [line width=1.pt,color=qqzzqq] (3.096192290249435,0.)-- (-1.9304443175641588,2.420700608291801);
\draw [line width=1.pt,color=qqzzqq] (0.4034034394724439,1.2967787643420936) -- (0.68658813857807,1.4257149603901347);
\draw [line width=1.pt,color=qqzzqq] (0.4034034394724439,1.2967787643420936) -- (0.4791598341072071,0.9949856479016672);
\draw [line width=1.pt,color=qqzzqq] (-1.9304443175641588,2.420700608291801)-- (5.026636607813593,2.4207006082917997);
\draw [line width=1.pt,color=qqzzqq] (1.7472934240362819,2.420700608291802) -- (1.548096145124717,2.1816638735979246);
\draw [line width=1.pt,color=qqzzqq] (1.7472934240362819,2.420700608291802) -- (1.548096145124717,2.659737342985679);
\draw [line width=1.pt,color=qqzzqq] (5.026636607813593,2.4207006082917997)-- (-1.2414767174295134,5.439264893538141);
\draw [line width=1.pt,color=qqzzqq] (1.7131093983218442,4.016411211111164) -- (1.9962940974274703,4.145347407159205);
\draw [line width=1.pt,color=qqzzqq] (1.7131093983218442,4.016411211111164) -- (1.7888657929566074,3.7146180946707372);
\draw [line width=1.pt,color=qqzzqq] (-1.2414767174295134,5.439264893538141)-- (4.337669007678949,5.439264893538139);
\draw [line width=1.pt,color=qqzzqq] (1.7472934240362834,5.43926489353814) -- (1.5480961451247186,5.200228158844263);
\draw [line width=1.pt,color=qqzzqq] (1.7472934240362834,5.43926489353814) -- (1.5480961451247186,5.678301628232018);
\begin{scriptsize}
\draw [fill=black] (0.,0.) circle (2.0pt);
\draw[color=black] (-0.22376235827664415,-0.535619954648519) node {$0$};
\draw [fill=black] (3.096192290249435,0.) circle (2.5pt);
\draw[color=black] (3.468027210884356,-0.5621795918367276) node {$6$};
\draw [fill=black] (5.026636607813593,2.4207006082917997) circle (2.0pt);
\draw[color=black] (5.353761451247169,2.1734630385487588) node {$5$};
\draw [fill=black] (4.337669007678949,5.439264893538139) circle (2.0pt);
\draw[color=black] (4.689770521541953,5.387179138322001) node {$4$};
\draw [fill=black] (1.548096145124718,6.7826523814585205) circle (2.0pt);
\draw[color=black] (1.582292970521543,7.086995918367351) node {$3$};
\draw [fill=black] (-1.2414767174295134,5.439264893538141) circle (2.0pt);
\draw[color=black] (-1.3658267573696152,5.652775510204086) node {$2$};
\draw [fill=black] (-1.9304443175641588,2.420700608291801) circle (2.0pt);
\draw[color=black] (-2.268854421768709,2.359380498866219) node {$1$};
\end{scriptsize}
\end{tikzpicture}   %
\definecolor{qqzzqq}{rgb}{0.,0.6,0.}
\definecolor{ffqqqq}{rgb}{1.,0.,0.}
\definecolor{ffffff}{rgb}{1.,1.,1.}
\begin{tikzpicture}[line cap=round,line join=round,>=triangle 45,x=1cm,y=1cm,scale=0.4]
\clip(-4.,-1.) rectangle (7.,9.);

\draw [line width=1.pt,color=ffqqqq] (0.,0.)-- (-2.189338564292883,2.1893385642928838);
\draw [line width=1.pt,color=ffqqqq] (-1.2355230288587165,1.2355230288587167) -- (-0.9256447860917102,1.263693778201172);
\draw [line width=1.pt,color=ffqqqq] (-1.2355230288587165,1.2355230288587167) -- (-1.2636937782011717,0.9256447860917114);
\draw [line width=1.pt,color=ffqqqq] (-2.189338564292883,2.1893385642928838)-- (-2.189338564292882,5.285530854542318);
\draw [line width=1.pt,color=ffqqqq] (-2.189338564292882,3.936631988329166) -- (-1.9503018295990042,3.737434709417601);
\draw [line width=1.pt,color=ffqqqq] (-2.189338564292882,3.936631988329166) -- (-2.4283752989867597,3.737434709417601);
\draw [line width=1.pt,color=ffqqqq] (-2.189338564292882,5.285530854542318)-- (0.,7.474869418835201);
\draw [line width=1.pt,color=ffqqqq] (-0.9538155354341654,6.521053883401034) -- (-0.9256447860917102,6.211175640634029);
\draw [line width=1.pt,color=ffqqqq] (-0.9538155354341654,6.521053883401034) -- (-1.2636937782011717,6.549224632743489);
\draw [line width=1.pt,color=ffqqqq] (0.,7.474869418835201)-- (3.0961922902494354,7.4748694188352);
\draw [line width=1.pt,color=ffqqqq] (1.7472934240362834,7.474869418835201) -- (1.5480961451247186,7.235832684141323);
\draw [line width=1.pt,color=ffqqqq] (1.7472934240362834,7.474869418835201) -- (1.5480961451247186,7.713906153529078);
\draw [line width=1.pt,color=ffqqqq] (3.0961922902494354,7.4748694188352)-- (5.285530854542318,5.285530854542317);
\draw [line width=1.pt,color=ffqqqq] (4.331715319108151,6.239346389976483) -- (4.0218370763411455,6.211175640634029);
\draw [line width=1.pt,color=ffqqqq] (4.331715319108151,6.239346389976483) -- (4.3598860684506064,6.549224632743489);
\draw [line width=1.pt,color=ffqqqq] (5.285530854542318,5.285530854542317)-- (5.285530854542317,2.1893385642928824);
\draw [line width=1.pt,color=ffqqqq] (5.285530854542317,3.538237430506036) -- (5.04649411984844,3.7374347094176006);
\draw [line width=1.pt,color=ffqqqq] (5.285530854542317,3.538237430506036) -- (5.524567589236195,3.7374347094176006);
\draw [line width=1.pt,color=ffqqqq] (5.285530854542317,2.1893385642928824)-- (3.096192290249435,0.);
\draw [line width=1.pt,color=ffqqqq] (4.050007825683601,0.9538155354341668) -- (4.0218370763411455,1.263693778201172);
\draw [line width=1.pt,color=ffqqqq] (4.050007825683601,0.9538155354341668) -- (4.3598860684506064,0.9256447860917114);
\draw [line width=1.pt,color=qqzzqq] (0.,0.)-- (3.096192290249435,0.);
\draw [line width=1.pt,color=qqzzqq] (1.7472934240362819,0.) -- (1.548096145124717,-0.23903673469387746);
\draw [line width=1.pt,color=qqzzqq] (1.7472934240362819,0.) -- (1.548096145124717,0.23903673469387746);
\draw [line width=1.pt,color=qqzzqq] (3.096192290249435,0.)-- (-2.189338564292883,2.1893385642928838);
\draw [line width=1.pt,color=qqzzqq] (0.2693925740599394,1.170898780558106) -- (0.5449022610722728,1.315510428848446);
\draw [line width=1.pt,color=qqzzqq] (0.2693925740599394,1.170898780558106) -- (0.3619514648842795,0.8738281354444388);
\draw [line width=1.pt,color=qqzzqq] (-2.189338564292883,2.1893385642928838)-- (5.285530854542317,2.1893385642928824);
\draw [line width=1.pt,color=qqzzqq] (1.7472934240362819,2.1893385642928833) -- (1.548096145124717,1.950301829599006);
\draw [line width=1.pt,color=qqzzqq] (1.7472934240362819,2.1893385642928833) -- (1.548096145124717,2.428375298986761);
\draw [line width=1.pt,color=qqzzqq] (5.285530854542317,2.1893385642928824)-- (-2.189338564292882,5.285530854542318);
\draw [line width=1.pt,color=qqzzqq] (1.3640618562063815,3.8136642078292646) -- (1.639571543218715,3.958275856119605);
\draw [line width=1.pt,color=qqzzqq] (1.3640618562063815,3.8136642078292646) -- (1.4566207470307224,3.5165935627155966);
\draw [line width=1.pt,color=qqzzqq] (-2.189338564292882,5.285530854542318)-- (5.285530854542318,5.285530854542317);
\draw [line width=1.pt,color=qqzzqq] (1.7472934240362834,5.285530854542317) -- (1.5480961451247186,5.04649411984844);
\draw [line width=1.pt,color=qqzzqq] (1.7472934240362834,5.285530854542317) -- (1.5480961451247186,5.524567589236195);
\draw [line width=1.pt,color=qqzzqq] (5.285530854542318,5.285530854542317)-- (0.,7.474869418835201);
\draw [line width=1.pt,color=qqzzqq] (2.4587311383528228,6.456429635100423) -- (2.7342408253651564,6.601041283390763);
\draw [line width=1.pt,color=qqzzqq] (2.4587311383528228,6.456429635100423) -- (2.5512900291771636,6.159358989986755);
\begin{scriptsize}
\draw [fill=black] (0.,0.) circle (2.0pt);
\draw[color=black] (-0.06440453514739236,-0.6949777777777706) node {$0$};
\draw [fill=black] (3.096192290249435,0.) circle (2.5pt);
\draw[color=black] (3.4149079365079387,-0.7215374149659792) node {$7$};
\draw [fill=black] (5.285530854542317,2.1893385642928824) circle (2.0pt);
\draw[color=black] (5.805275283446716,1.8813070294784642) node {$6$};
\draw [fill=black] (5.285530854542318,5.285530854542317) circle (2.0pt);
\draw[color=black] (5.645917460317464,5.387179138322001) node {$5$};
\draw [fill=black] (3.0961922902494354,7.4748694188352) circle (2.0pt);
\draw[color=black] (3.468027210884356,7.67130793650794) node {$4$};
\draw [fill=black] (0.,7.474869418835201) circle (2.0pt);
\draw[color=black] (-0.14408344671201823,7.830665759637192) node {$3$};
\draw [fill=black] (-2.189338564292882,5.285530854542318) circle (2.0pt);
\draw[color=black] (-2.534450793650795,5.334059863945583) node {$2$};
\draw [fill=black] (-2.189338564292883,2.1893385642928838) circle (2.0pt);
\draw[color=black] (-2.5610104308390036,1.8281877551020471) node {$1$};
\end{scriptsize}
\end{tikzpicture}   \caption{The graph $\Gamma_m$ for $m=6$ and $m=7$. The edges in $E_1$ are marked in red, the edges of $E_2$ in green.}
  \label{fig:Graph}
\end{figure}
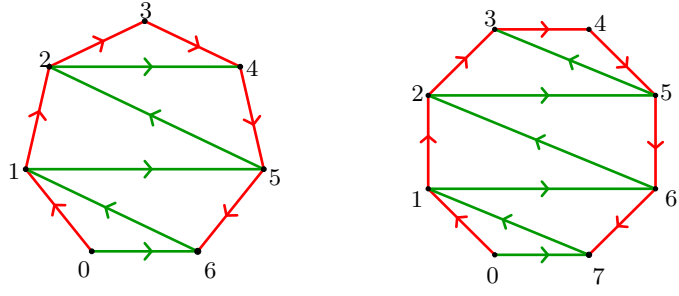

Consider the vector space $\C^E = \{w \colon E \to \C\}$ of complex weight functions on the edges of $\Gamma$. For each vertex $k \in \{0,1,\ldots,m\}$, consider the linear map
\begin{align*}
  \Psi_k \colon \C^E &\to \C^4, \\
  w &\mapsto \big(\sum_{(i,k) \in E_1} w(i,k), \sum_{(k,i) \in E_1} w(k,i), \sum_{(i,k) \in E_2} w(i,k), \sum_{(k,i) \in E_2} w(k,i)\big),
\end{align*}
extracting from a weight function the total weights of incoming and outgoing edges at vertex $k$ from $E_1$ and $E_2$, respectively. For every $w \in \C^E$, let $Z_w \subseteq \C^4$ denote the vector space spanned by $\Psi_0(w),\ldots,\Psi_m(w)$.

We now define the linear space $L$ which we will check to satisfy the properties of \cref{thm:graphConstr}: Let $W_1 \subseteq \C^2 \otimes \C^2 \otimes \C^{m+1}$ be the $m$-dimensional vector space with the basis 
  \[u_{i,j} := e_0 \otimes e_0 \otimes e_j + e_0 \otimes e_1 \otimes e_i + e_1 \otimes e_0 \otimes e_i \qquad \text{for all } (i,j) \in E_1\]
Similarly, let $W_2 \subseteq \C^2 \otimes \C^2 \otimes \C^{m+1}$ be the $(m-1)$-dimensional vector space with the basis
  \[v_{i,j} := e_1 \otimes e_0 \otimes e_j + e_1 \otimes e_1 \otimes e_i + e_0 \otimes e_0 \otimes e_i \qquad \text{for all } (i,j) \in E_2\]
Define $W := W_1+W_2$, which is a vector space of dimension $2m-1$. Denote by $L_1,L_2,L \subseteq \P(\C^2 \otimes \C^2 \otimes \C^{m+1})$ the corresponding linear subspaces $L_i := \P W_i$ and $L := \P W$.

We may identify elements of $W$ with elements of $\C^E$ under the linear isomorphism
  \[\Phi \colon \C^E \xrightarrow{\cong} W, \qquad 
  w \mapsto \sum_{(i,j) \in E_1} w(i,j) \: u_{i,j} + \sum_{(i,j) \in E_2} w(i,j) \: v_{i,j}.\]
This allows for a combinatorial reformulation of the condition that a tensor in $W$ is of border rank $\leq 2$:

\begin{lemma} \label{lem:combinatorialRank2}
  A tensor $t \in W \subseteq \C^2 \otimes \C^2 \otimes \C^{m+1}$ is of border rank~$2$ if and only if $w := \Phi^{-1}(t) \in \C^E$ satisfies $\dim Z_w \leq 2$.
\end{lemma}

\begin{proof}
  By \cite[Theorem~5.1]{LM04}, a tensor $t \in W \subseteq \C^2 \otimes \C^2 \otimes \C^{m+1}$ is of border rank $\leq 2$ if and only if the induced linear map 
    \[\varphi_t \colon (\C^{m+1})^* \to \C^2 \otimes \C^2,
    \qquad \ell \mapsto (\id \otimes \id \otimes \ell)(t)\]
  has image $\im \varphi_t \subseteq \C^2 \otimes \C^2$ of dimension at most~$2$.
  Let $w := \Phi^{-1}(t) \in \C^E$. Composing $\varphi_t$ with the isomorphism
  \[\psi := (e_0^* \otimes e_0^* - e_1^* \otimes e_1^*, \ e_0^* \otimes e_1^*, \ e_1^* \otimes e_0^* - e_0^* \otimes e_1^*, \ e_1^* \otimes e_1^*) \colon \C^2 \otimes \C^2 \xrightarrow{\cong} \C^4\]
  gives the linear map 
  \[\psi \circ \varphi_t \colon (\C^{m+1})^* \to \C^4, \qquad e_k^* \mapsto \Psi_k(w),\]
  whose image is precisely $Z_w$.
\end{proof}

Based on \cref{lem:combinatorialRank2}, our proof of \cref{thm:graphConstr} will become very combinatorial. The main graph-theoretical observations are the content of the next two Lemmas:

\begin{lemma}
  Let $w \in \C^E$. If $Z_w \subseteq \C^4$ lies in one of the coordinate hyperplanes, then $\Phi(w) \in W_1$ or $\Phi(w) \in W_2$.
\end{lemma}

\begin{proof}
  If $Z_w \subseteq V(e_0^*)$, then the first coordinate of $\Psi_k(w)$ is zero for all $k \in \{0,1,\ldots,m\}$. Since every vertex of $\Gamma$ has at most one incoming edge from $E_1$ and since every edge from $E_1$ is an incoming edge at some vertex, this shows that $w(i,j) = 0$ for all $(i,j) \in E_1$. Hence, $\Phi(w) \in W_2$. The other three cases $Z_w \subseteq V(e_i^*)$ for $i \in \{1,2,3\}$ follow by the same argument.
\end{proof}

\begin{lemma}
 Let $w \in \C^E$. If $\dim Z_w \leq 2$, then $Z_w \subseteq \C^4$ lies in one of the coordinate hyperplanes.
\end{lemma}

\begin{proof}
  By induction on $m$. Assume that $Z_w$ is not contained in any coordinate hyperplane. Consider the special vertex $k := \lceil m/2 \rceil$ in $\Gamma$, which has no adjacent edges in $E_2$. 
  
  First, we show that $\Psi_k(w) = 0$: Along the directed path of edges from $E_2$, let $(i,j)$ is the first edge with non-zero weight. Then both the vertices $i$ and $k$ have no weight on incoming edges from $E_2$, hence $\Psi_i(w)$ and $\Psi_k(w)$ lie in the coordinate hyperplane $V(e_2^*) \subseteq \C^4$. Since $\dim Z_w \leq 2$ and $Z_w \not\subseteq V(e_2^*)$, the vectors $\Psi_i(w)$ and $\Psi_k(w)$ must be proportional. But $\Psi_i(w) \notin V(e_3^*)$, while $\Psi_k(w) \in V(e_3^*)$, so we conclude that $\Psi_k(w) = 0$.
  
  Secondly, we may assume that the edge in $E_2$ between the vertices $k-1$ and $k+1$ has non-zero weight: Otherwise, we may delete this edge as well as the the vertex $k$ and its incident edges to obtain a weighted graph which can be viewed as a subgraph of the graph $\Gamma$ for the case $m-1$. This case is covered by the induction hypothesis.
  
  In particular, $\Psi_{k-1}(w) \neq 0$ and $\Psi_{k+1}(w) \neq 0$. One of the vertices $k-1$ and $k+1$ has no outgoing edge in $E_1$, while the other has an outgoing edge in $E_2$ with non-zero weight. Hence, $\Psi_{k-1}(w)$ and $\Psi_{k+1}(w)$ are not proportional. Because of $\dim Z_w \leq 2$, they must form a basis of $Z_w$.
  
  We now need to distinguish between even and odd $m$.
  
  \emph{Case 1: $m$ is even.} Then the last edge of the directed path formed by $E_2$-edges is $(k-1, k+1) \in E_2$. The vertex $0$ has no incoming edge in $E_1$, so $\Psi_0(w) \in V(e_0^*)$. Since also $\Psi_{k+1}(w) \in V(e_0^*)$ and $Z_w \not\subseteq V(e_0^*)$ and $\dim Z_w = 2$, the vector $\Psi_0(w)$ must be a multiple of $\Psi_{k+1}(w)$. But the vertex $0$ has no incoming edge in $E_2$, while the vertex $k+1$ has an incoming edge in $E_2$ with non-zero weight, so in fact, we must have $\Psi_0(w) = 0$.
  
  If also $\Psi_m(w) = 0$, then we may delete the vertices $0$ and $m$ and their incident edges to obtain the graph $\Gamma$ for the case of replacing $m$ by $m-2$, so this case is already covered by the induction hypothesis. So, we may assume $\Psi_m(w) \neq 0$.
  
  Since the vertex $m$ has no outgoing edge in $E_1$, the vectors $\Psi_m(w)$ and $\Psi_{k-1}(w)$ are both non-zero vectors in $V(e_1^*)$, so they must be proportional by $\dim Z_w \leq 2$ and $Z_w \not\subseteq V(e_1^*)$. Since $\Psi_0(w) = 0$, we have $\Psi_m(w) \in V(e_2^*)$, so we must also have $\Psi_{k-1}(w) \in V(e_2^*)$. This means that the edge $(k+2,k-1) \in E_2$ has weight zero.
  
  This in turn implies $\Psi_{k+2}(w) \in V(e_3^*)$. In particular, $\Psi_{k+2}(w)$ is proportional to $\Psi_{k+1}(w)$. Then $\Psi_{k+2}(w) \in V(e_0^*)$, so the edge $(k+1,k+2)\in E_1$ has weight zero. But then both $\Psi_{k-1}(w)$ and $\Psi_{k+1}(w)$ lie in $V(e_1^*)$, contradicting $Z_w \not\subseteq V(e_1^*)$, since they form a basis. This concludes Case~1.
  
  \emph{Case 2: $m$ is odd.} Here, the last edge along the directed path of edges from $E_2$ is $(k+1,k-1)$. We argue similar to Case~1: The vector $\Psi_m(w)$ is proportional to $\Psi_{k-1}(w)$, since the vertex $m$ has no outgoing edge in $E_1$. The vertex $0$ has no incoming edge in $E_1$, so $\Psi_0(w)$ is proportional to $\Psi_{k+1}(w)$. 
  
  If one of $\Psi_0(w)$ and $\Psi_n(w)$ is zero, then so is the other, because the edge in $E_2$ between them has weight zero, while the edge $(k+1,k-1)$ does not. But this case is covered by the induction hypothesis, as previously in Case~1.
  
  Hence, $\Psi_0(w) \neq 0$ and $\Psi_m(w) \neq 0$. The edge $(m,2) \in E_2$ has weight zero, since $\Psi_m(w)$ is proportional to $\Psi_{k-1}(w)$. But this implies that $\Psi_1(w)$ must be proportional to $\Psi_{k+1}(w)$. Then the edge $(0,1)\in E_1$ must have weight zero. But since $\Psi_0(w)$ is a non-zero multiple of $\Psi_{k+1}(w)$, this implies that both $\Psi_{k-1}(w)$ and $\Psi_{k+1}(w)$ lie in $V(e_1^*)$, a contradiction.
  This concludes Case~2 and therefore the proof.
\end{proof}

\begin{proof}[Proof of \cref{thm:graphConstr}]
  Combining the previous Lemmas, we have established the following: A tensor $t \in W$ is of border rank $\leq 2$ if and only if $t \in W_1$ or $W_2$. In other words, $\sigma_2(Y_{\nVec,\dVec}) \cap L = L_1 \sqcup L_2$.
  
  It only remains to show that $L_k \cap \sigma_2^{\circ}(X) = \emptyset$ for $k = 1,2$. For all $\lambda_{i,j} \in \C$, the vector $\sum_{(i,j) \in E_1} \lambda_{i,j} u_{i,j} \in W_1$ is the tangent vector to the Segre variety at the point $e_0 \otimes e_0 \otimes (\sum_{(i,j) \in E_1} \lambda_{i,j} e_i)$ in the direction $e_1 \otimes e_1 \otimes (\sum_{(i,j) \in E_1} \lambda_{i,j} e_j)$. Such a tangent vector is of rank~$3$ unless $\sum_{(i,j) \in E_1} \lambda_{i,j} e_i$ is proportional to $\sum_{(i,j) \in E_1} \lambda_{i,j} e_j$, which is the case if and only if $\lambda_{i,j}=0$ for all $(i,j) \in E_1$. Hence $L_1 \cap \sigma_2^{\circ}(X) = \emptyset$. The same argument proves the claim for $L_2$.
\end{proof}

We gave a geometric construction of an injection $\P^1 \times \P^1 \times \P^m \injTo \P^{2m+4}$. Choosing appropriate bases, we arrive at the following explicit description:

\begin{cor} \label{cor:P1P1Pm}
  The following morphism is injective:
  \begin{align*}
    &\P^1 \times \P^1 \times \P^m \to \P^{2(m+2)}, \\
    [x_0:x_1] \times &[y_0:y_1] \times [z_0:\ldots:z_m] \\ &\mapsto \Big[x_0y_1 z_m : x_1 y_1 z_{\lceil m/2 \rceil} : x_1 y_1 z_{\lceil m/2 \rceil + (-1)^m} : x_0 y_0 z_{i+1} - x_0 y_1 z_i :  \\
    &\phantom{\mapsto  \Big[~} x_0 y_0 z_{i+1} - x_1 y_0 z_i : x_1 y_0 z_{m-j+\lfloor 2j/m \rfloor} - x_0 y_0 z_j : x_1 y_0 z_{m-j+\lfloor 2j/m \rfloor} - x_1 y_1 z_j \mid \\
    &\phantom{\mapsto  \Big[~} i \in \{0,1,\ldots,m-1\}, \; j \in \{0,1,\ldots,m\} \setminus \{\lceil m/2 \rceil, \lceil m/2 \rceil + (-1)^m\}\Big].
  \end{align*}
  In particular, $\injDim{\P^1 \times \P^1 \times \P^m, \O(1,1,1)} \leq 2(m+2)$.
\end{cor} 

\renewcommand*{\bibfont}{\small}
\printbibliography

\end{document}